\documentclass[a4paper]{amsart}

\usepackage{amsmath}
\usepackage{xspace}
\usepackage[psamsfonts]{amssymb}
\usepackage[utf8]{inputenc}
\usepackage{float}\restylefloat{figure}
\usepackage{color}
\usepackage{tikz}\usetikzlibrary{arrows}
\graphicspath{{img/}}

\usepackage{hyperref}

\theoremstyle{plain}
\newtheorem{definition}{Definition}
\newtheorem{proposition}{Proposition}
\newtheorem{lemma}{Lemma}
\newtheorem{corollary}{Corollary}
\newtheorem{theorem}{Theorem}
\newtheorem{prelemma}{Pre Lemma}
\theoremstyle{remark}
\newtheorem*{remark}{Remark}
\newtheorem*{example}{Example}

\renewcommand{\epsilon}{\varepsilon}

\newcommand{\cal}[1]{\mathcal{#1}}
\newcommand{\ov}[1]{\overline{#1}}
\newcommand{\setof}[2]{\big\{{#1}\,\big|\,{#2}\big\}}

\def\bEA{\begin{eqnarray*}}
\def\eEA{\end{eqnarray*}}
\def\bEAn{\begin{eqnarray}}
\def\eEAn{\end{eqnarray}}

\def\cal{\mathcal}
\def\ov{\overline}

\def\tend{\longrightarrow}

\def\wh{\widehat}
\def\on{\operatorname}

\def\cal{\mathcal}

\def\C{{\mathbb C}}
\def\D{{\mathbb D}}
\def\H{{\mathbb H}}
\def\I{{\mathbb I}}
\def\N{{\mathbb N}}

\def\R{{\mathbb R}}

\def\Z{{\mathbb Z}}

\def\Re{{\on{Re}\,}}
\def\Im{{\on{Im}\,}}

\newcommand{\sq}{\on{Sq}}

\begin{document}

\email{arnaud.cheritat@math.univ-toulouse.fr}

\address{CNRS / Institut de mathématiques de Toulouse\\ Université Paul Sabatier\\ 118 Route de Narbonne\\ Toulouse, France}

\title[Proof of a convergence]{Proof of a convergence of affine Riemann surfaces to a richer one}

\begin{abstract}
We give a proof of a phenomenon conjectured in our former article: ``Beltrami forms, affine surfaces and the Schwarz-Christoffel formula: a worked out example of straightening''. We also start an abstract discussion of the notion of limits of Riemann surfaces.
\end{abstract}

\author{Arnaud Chéritat}

\maketitle

Reminder: a topological space satisfies \emph{Hausdorff's separation axiom} (also called $T_2$ axiom) if every pair of distinct points has a pair of disjoint neighborhoods.

\section*{Prologue}

In the preprint \cite{C}, the author studied a curious enrichment phenomenon. Though it is motivated by the study of a particular case of the Beltrami equation, it was reformulated there in terms of uniformization of a Riemann surface depending on a parameter $K$. This surface was defined by gluing polygonal pieces along their boundaries\footnote{The gluings extend uniquely as holomorphic maps defined in a neighborhood of the segments or lines.}, vertices excluded, by complex-affine maps, i.e.\ maps of the form $z\mapsto az+b$.
There were two pieces: a rectangle and the complement of a square in $\C$. The ratio $K\geq 1$ of the lengths of the sides of the rectangle is the parameter. See Section~\ref{subsec:fam} for illustrations and more details.
It can be uniformized, as a Riemann surface, to the complex plane minus four points, corresponding to the four corners of the square. We gave in \cite{C} an integral formula to express the inverse of this uniformization, i.e.\ mapping back to the polygonal pieces. This formula is nothing but Schwarz-Christoffel's formula applied to a more general setting than the usual one. The image of the rectangle is a deformed square.
When the rectangle gets more and more flat, computer experiments showed that its (normalized) image tends to a non-trivial limit, with some enrichments\footnote{Reminiscent of parabolic enrichment in complex dynamics.}.

Since the change of coordinates are complex-affine, what we defined is not only a Riemann surface but is somewhat more rigid. Let us call affine Riemann surface this kind of manifold, with changes of charts that are complex-affine. In \cite{C} we defined another affine Riemann surface by gluing infinitely many polygonal pieces, that we presented as a candidate for being a natural limit. In particular, we stated the conjecture that it is conformally equivalent to $\C$ minus two points, and that these two points union the image of some specific pieces is equal to the Hausdorff limit as $K\tend+\infty$ of the image of the rectangle, see Figure~\ref{fig:rects}.

\begin{figure}
\begin{tikzpicture}
\node at (-1.65,0) {\includegraphics[scale=0.25,trim=100 150 100 150]{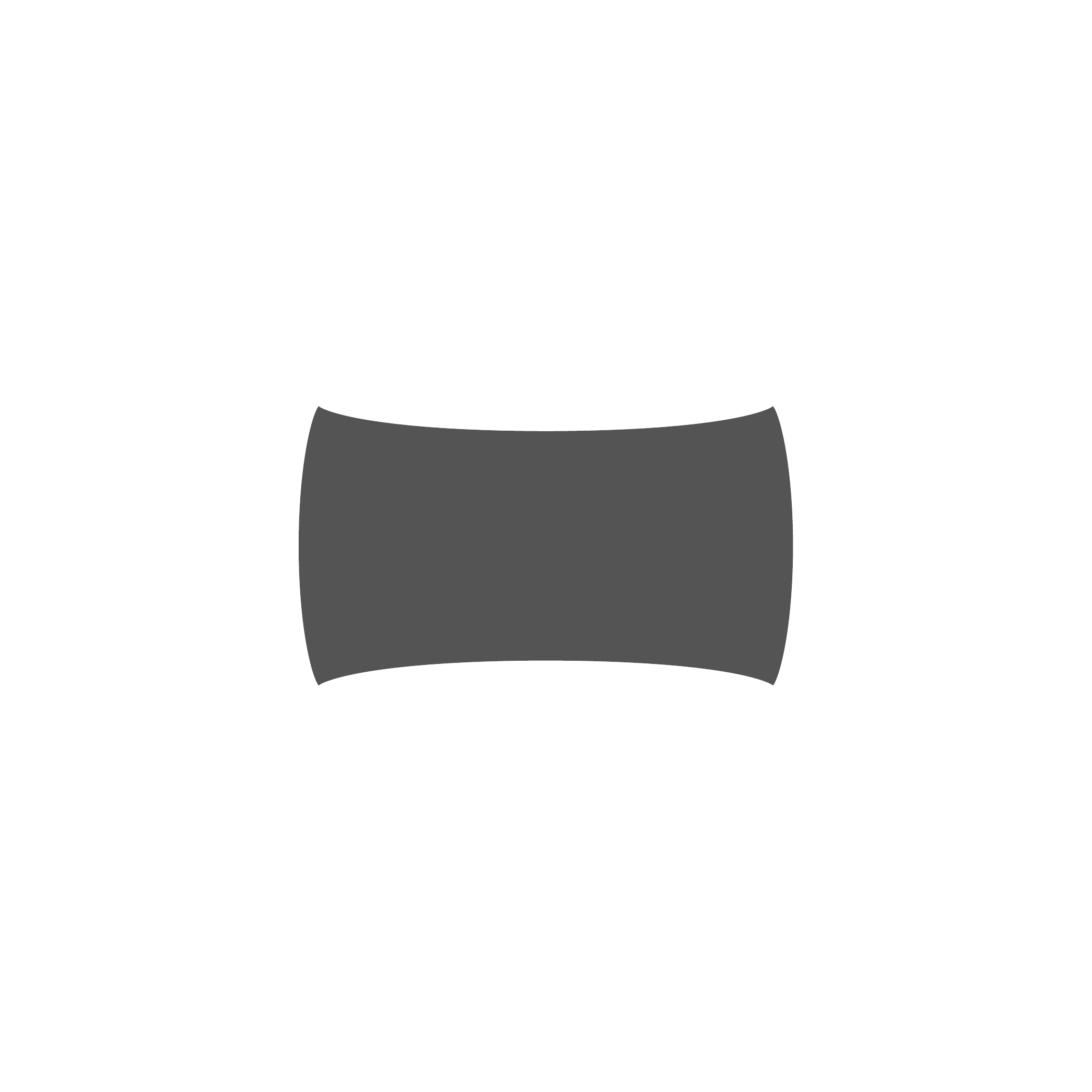}};
\node at (2.2,0) {\includegraphics[scale=0.25,trim=100 150 100 150]{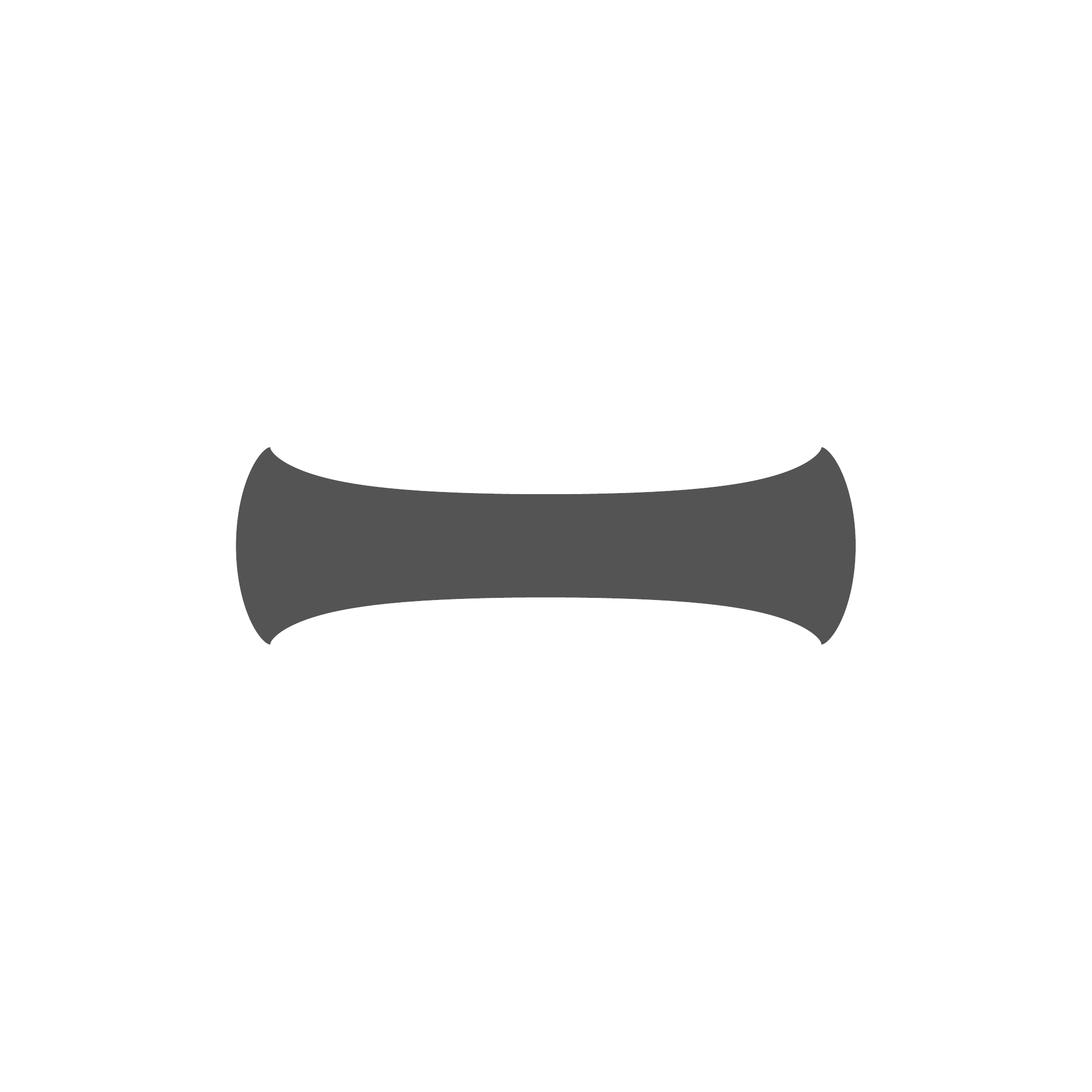}};
\node at (6.5,0) {\includegraphics[scale=0.25,trim=80 250 80 250]{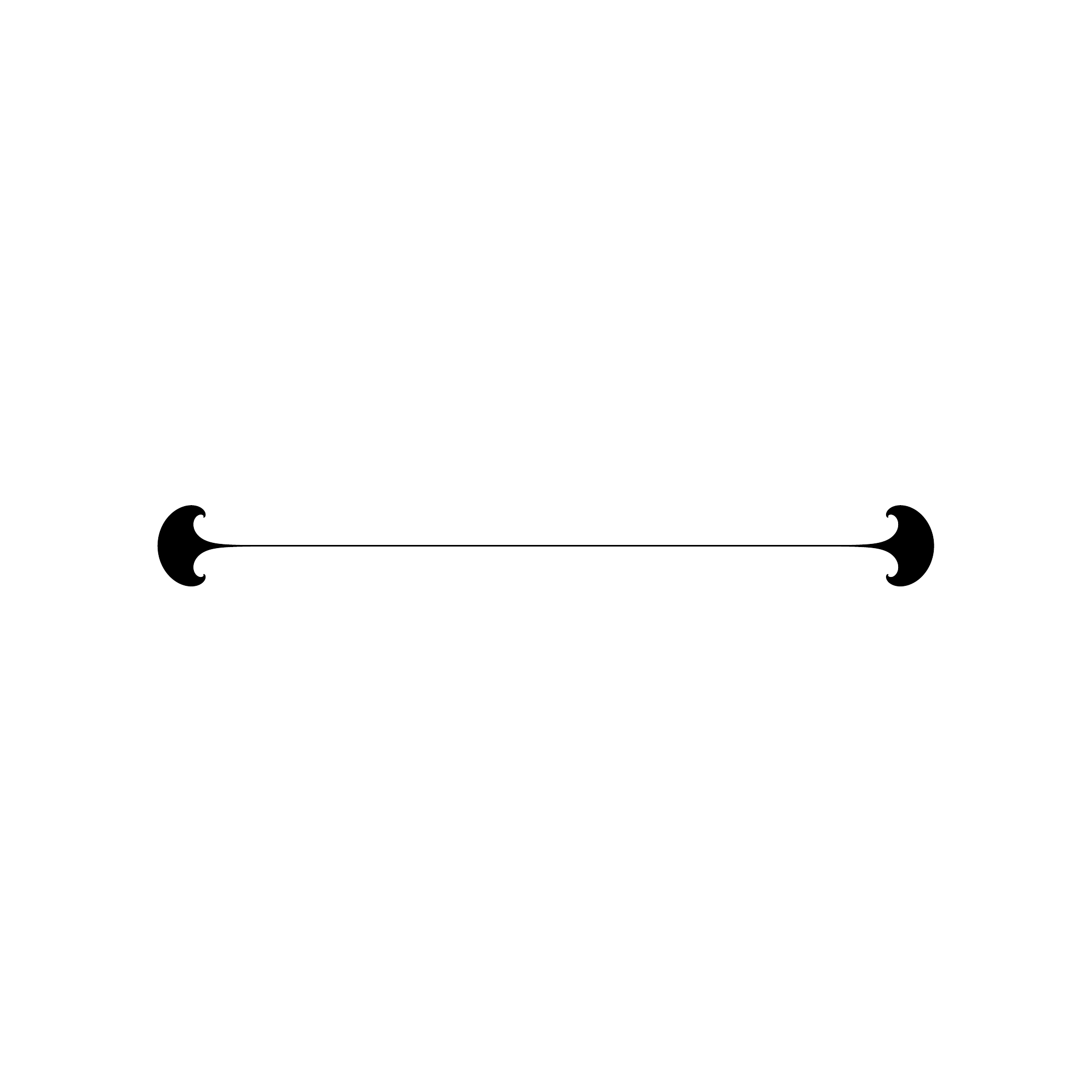}};
\node at (2.8,-1.8) {\includegraphics[width=11.5cm,trim=20 25 20 25]{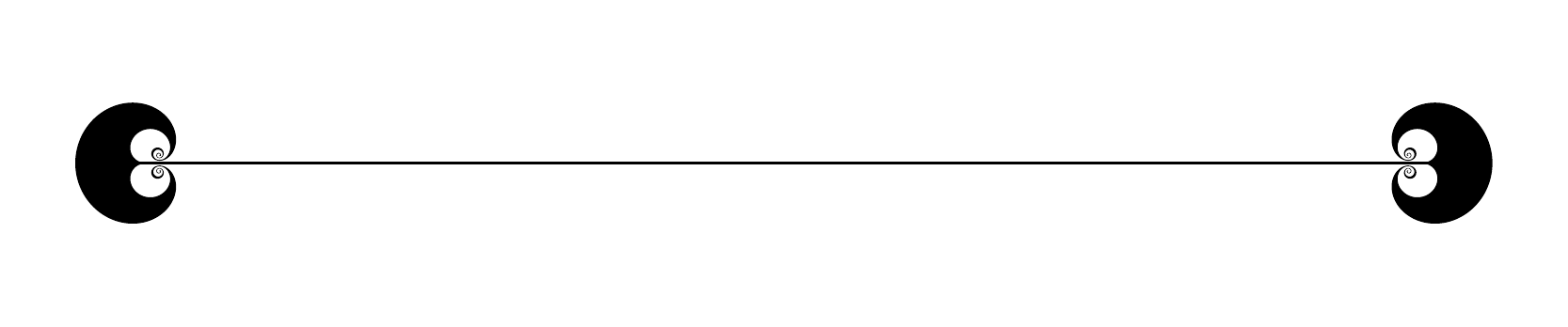}};
\node at (0.2,-5) {\includegraphics[width=6cm]{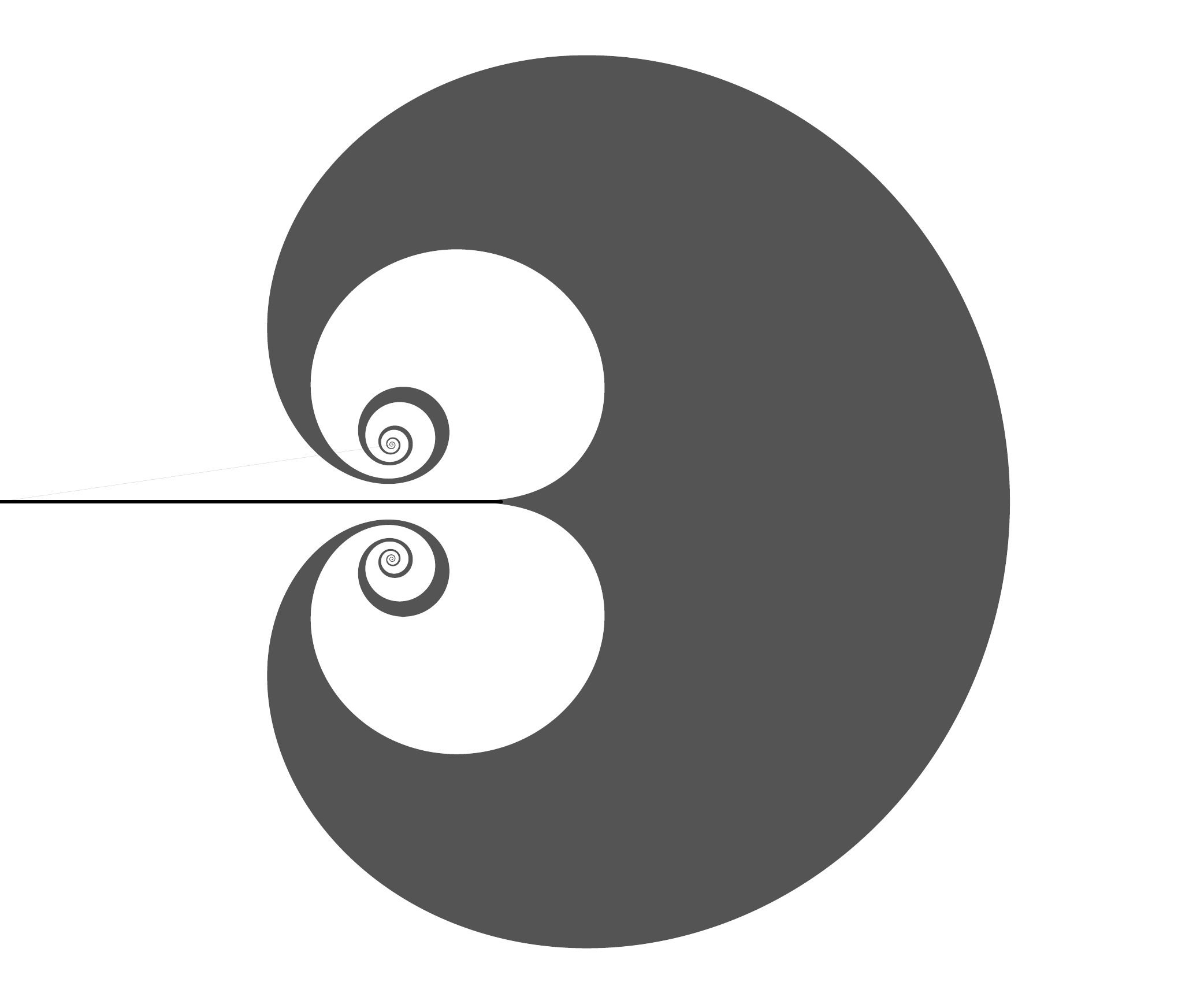}}; 
\node at (6,-5) {\includegraphics[width=6.45cm]{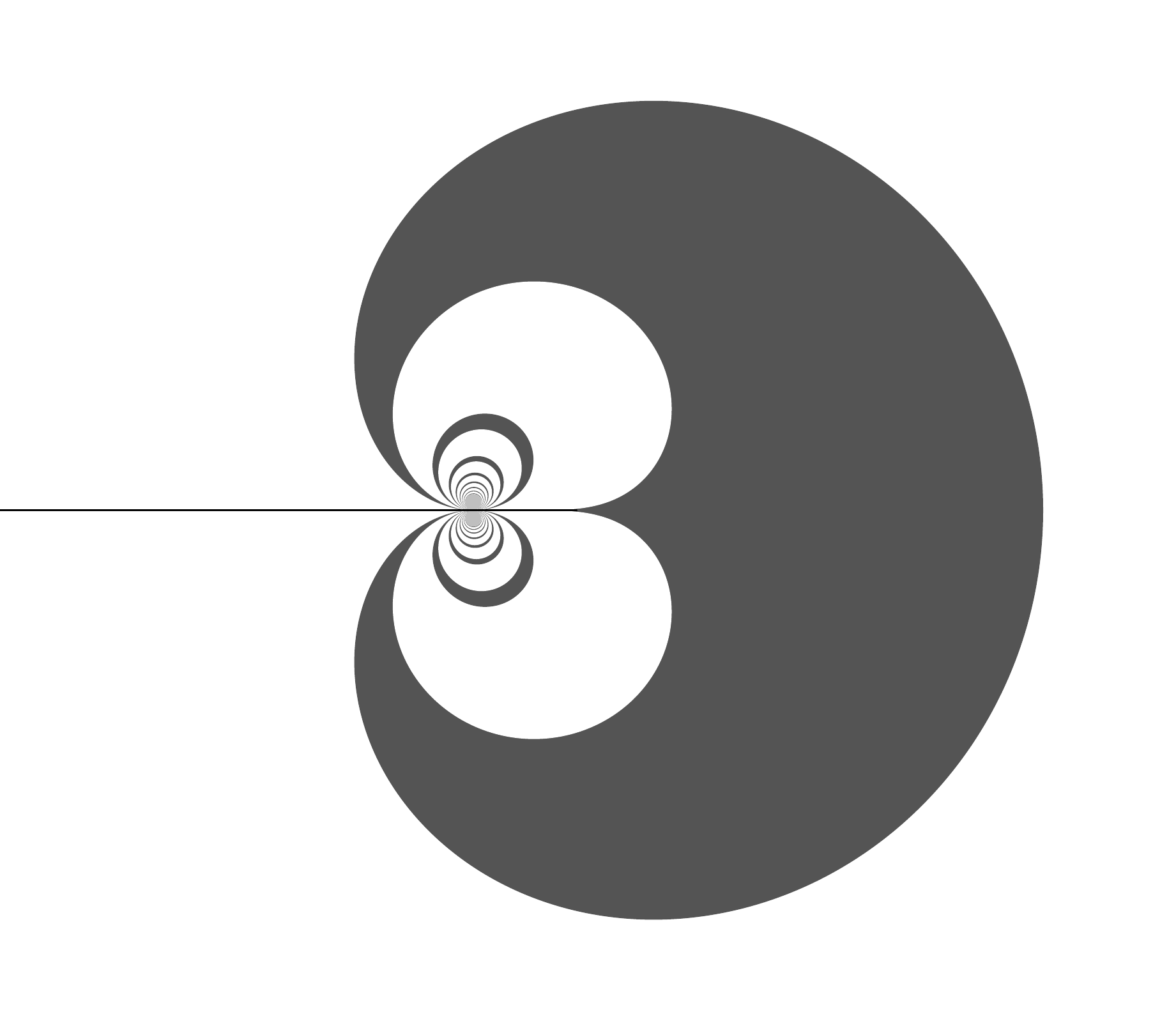}}; 
\end{tikzpicture}
\caption{Top: image of the rectangle for $K=2$, $K=5$, $K=1000$ and $K=10^{20}$. Bottom: zoom on the right bulb of the image of the rectangle for $K=10^{20}$, and conjectured limit for $K\tend +\infty$. Some of them are shown in gray instead of black to save ink in case this page gets printed.}
\label{fig:rects}
\end{figure}

To an affine Riemann surface, we can associate a Riemann surface by using exactly the same atlas. Now one may want to go in the opposite way. So start from a Riemann surface. It turns out that the data of a structure of affine Riemann surface that is compatible with the given Riemann surface, is equivalent to the data of a holomorphic Christoffel symbol (a connection) in complex dimension one. More practically stated, the expression of the symbol in a Riemann surface chart takes the form of a holomorphic function $\zeta(z)$ and the affine Riemann surface charts $\phi$ expressed in this chart are precisely the solutions of $\phi''/\phi'= \zeta$. Expressed in the normalized global chart $\C\setminus\{\text{a finite set}\}$ conformally equivalent to the affine Riemann surfaces mentioned above, the symbol turns out to be very explicit:
\[\zeta(z)=\frac{\log K}{2\pi i} \left(\frac{-1}{z-z_1}+\frac{1}{z-z_2}+\frac{-1}{z-z_3}+\frac{1}{z-z_4}\right)\]
where the $z_i$ depend on $K>1$, which is the ratio between the lengths of the long side of the rectangle and the short side.
We were able to prove in \cite{C} that when $K\tend+\infty$, then $z_1$, $z_4$ tend to a positive real $x_\infty$, $z_2$, $z_3$ to $-x_\infty$, and that the symbol tends to
\[\zeta(z) = \frac{-\tau}{(z-x_0)^2}+\frac{\tau}{(z+x_0)^2}\]
where $\tau>0$ is some constant.
It is natural to conjecture that our candidate limit affine Riemann surface is isomorphic the affine Riemann surface defined over $\C\setminus\{x_\infty,-x_\infty\}$ by this function $\zeta$.

In the present article we define a notion of being a limit and prove the conjectures stated above. In doing so, we also give a proof of the convergence of $z_i$ and $\zeta$ that is independent of the one given in \cite{C}. We added a section with a few abstract generalities about limits of manifolds, and an illustration with Riemann surfaces.

It is striking that the natural limit of a sequence of affine surfaces that can be defined by using a bounded finite number of polygonal pieces\footnote{One piece is engouh if we allow for slits with two distinct sides.}, may require infinitely many polygonal pieces.

\section{Matter}

\subsection{Embeddings}

Let $\cal M_\infty$ and $\cal M_n$ be a sequence of Riemann surfaces or of affine Riemann surfaces of some kind. 
The charts of $\cal M_n$ map some open subsets of $\cal M_n$ homeomorphically to open subsets of $\C$, so that the changes of charts belong to a set of maps between open subsets of $\C$, called the set of legal changes: holomorphic maps for Riemann surfaces, orientation preserving similitudes for affine Riemann surfaces. Let
\[  \I=\{0\}\cup\setof{1/n}{n>0}
.\]
We will consider the set $\I\times\C$ on which we put the topology induced by its inclusion in $\R^3$. The leaf number $n$ is the subset
\[L_n=\{1/n\}\times \C \]
(by convention $1/\infty=0$).

\begin{definition}\label{def:1}
An \emph{embedding of $\cal M_\infty$ at the limit} of the sequence $\cal M_n$ is a topology and an atlas over a subset $\cal U$ of the disjoint union $\cal M=\cal M_\infty\coprod \cal M_1\coprod\cal M_2\coprod\cdots$, such that $\cal U$ contains $\cal M_\infty$,
such that the topology satisfies Hausdorff's separation axiom,
such that the charts are homeomorphisms from open subsets of $\cal M$ to open subsets of $\I\times\C$, with $\cal M_n$ being mapped to the leaf $L_n$ for all $n\in\{\infty,1,2,\ldots\}$ by a map compatible with the original atlas of $\cal M_n$, and
such that the coordinate changes of the new atlas are continuous (as maps between open subsets of $\I\times\C$). 
\end{definition}
It is understood that the atlas covers $\cal U$. As the charts are leafwise compatible with $\cal M_n$, it follows that the coordinate changes leafwise belong to the set of legal changes. The continuity condition for the coordinate changes means that for any pair of charts on $\cal M$, the coordinate change on leaf $L_n$ converges to the coordinate change on leaf $L_\infty$ as $n\tend +\infty$.

For limits when $t\to+\infty$ of a family parameterized by a real $t>1$, replace in the definition the set $\I$ by $[0,1[$ (or $[0,1]$) and map $\cal M_t$ to leaf $1/t$. The continuity condition is then stronger. It can be relaxed, if necessary, to continuity of the coordinate change at every point of $L_0$ only (for the discrete case, this makes no difference).

\subsection{The family and its candidate limit}\label{subsec:fam}

We now recall the definition of the affine Riemann surfaces $\cal A_K$ and $\cal A_\infty$. They are defined by gluing polygonal patches, subsets of $\C$, along neighborhoods of their boundaries by affine maps. 	
By an abuse of language, we will often use the same name or symbol to denote the subset of $\C$ and the corresponding subset of the surface.

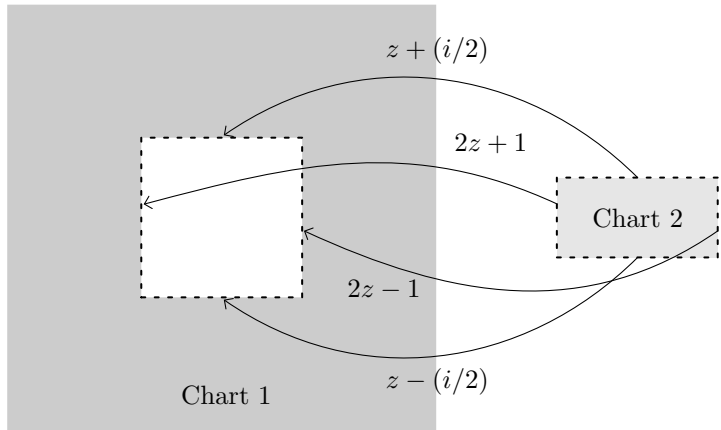
\begin{figure}
\begin{tikzpicture}[x=1pt,y=1pt,>=angle 90]
\node at (142,82) {\includegraphics{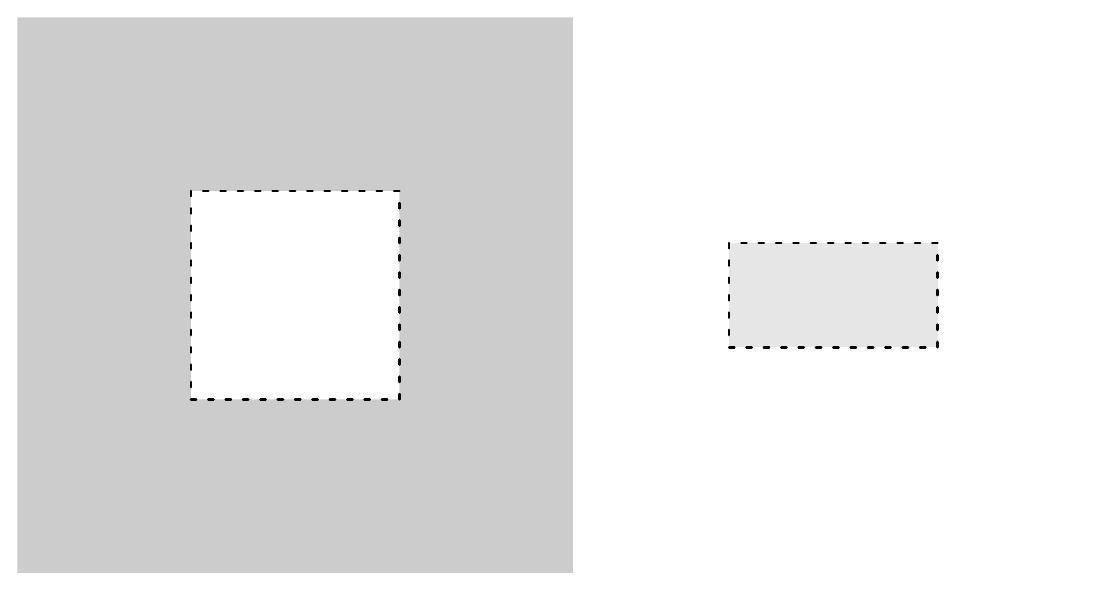}};
\node at (71,15) {Chart 1};
\node at (225,82) {Chart 2};
\draw (70,51) edge[out=-35,in=225,<-] (225,67);
\node at (150,145) {$z+(i/2)$};
\draw (70,113) edge[out=35,in=-225,<-] (225,97);
\node at (150,20) {$z-(i/2)$};
\draw (40,87) edge[out=15,in=-205,<-] (195,87);
\node at (170,110) {$2z+1$};
\draw (100,77) edge[out=-25,in=215,<-] (255,77);
\node at (130,55) {$2z-1$};
\end{tikzpicture}
\caption{The definition of $\cal A_K$, illustrated for $K=2$.}
\label{fig:glue}
\end{figure}

Let Chart~1 denote the complement in $\cal \C$ of the closed square $\sq=[1,1]\times[1,1]$.\footnote{We identify $\C$ with $\R^2$.}
Let Chart~2 denote the open rectangle $]-1,1[\times ]-\frac1K,\frac1K[ \subset\C$.
To each segment in the boundary of one chart, upper, lower, left or right, we associate the segment in the boundary of the other chart with the same adjective: upper, lower, left or right. The same can be done for corners.
Now glue the closure of Chart~1 to the closure of Chart~2 by identifying each pair of corresponding segments with the unique similitude respecting the position of corresponding corners. See Figure~\ref{fig:glue}. We then get a topological surface homeomorphic to the plane. Since the gluing were similitudes, if one removes the four corners, it is possible to define an atlas of affine Riemann surface for which Chart~1 and Chart~2 are charts. (For instance, use a small enough neighborhood of the closure of Chart~1 minus the corners, by adding tabs to the segments, do the same for Chart~2; the coordinate changes will be the same set of $4$ similutudes.)

\begin{figure}[htbp]%
\begin{picture}(350,290)
\put(0,0){\includegraphics[width=350pt]{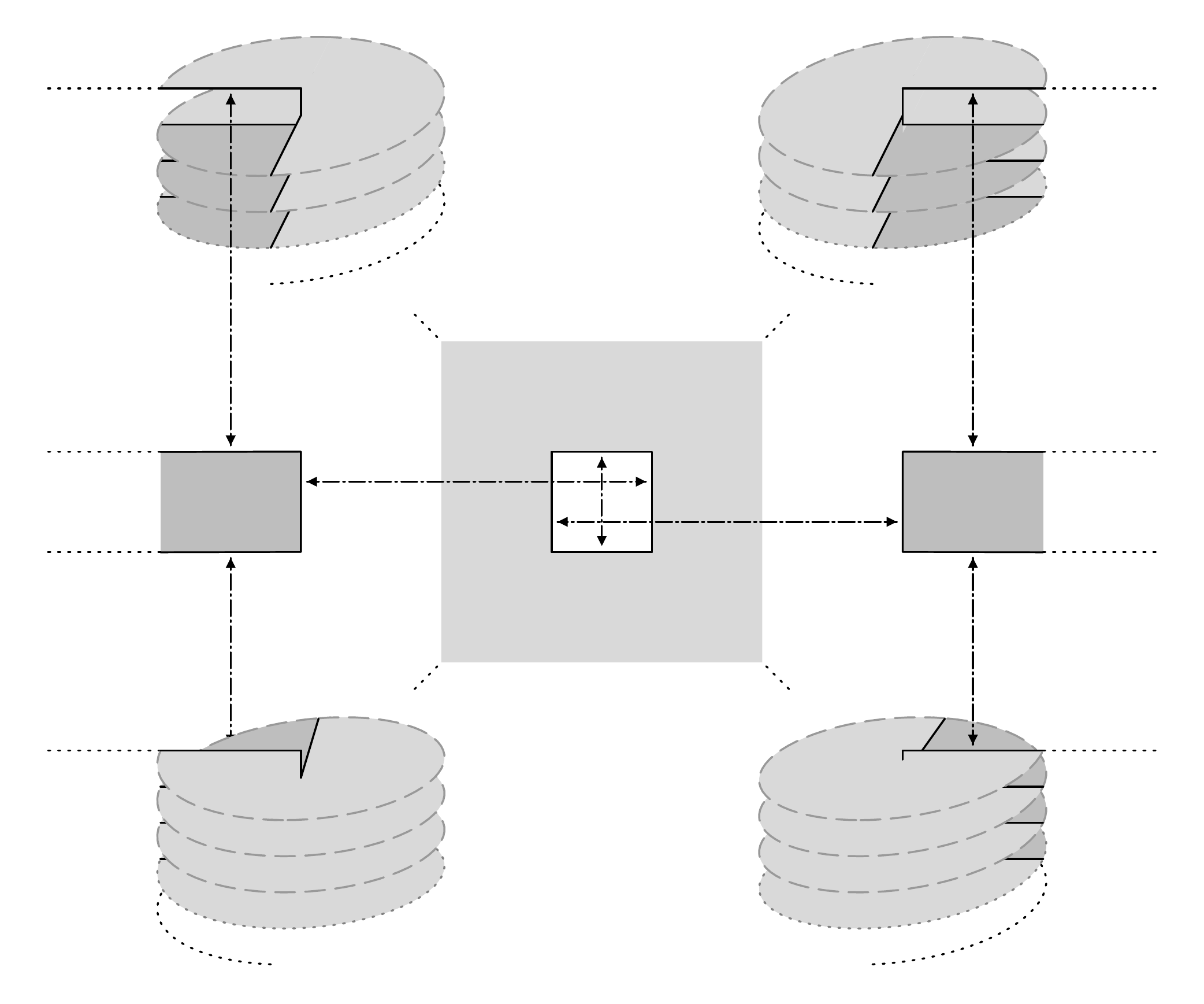}}
\put(138,110){Chart 1}
\put(62,143){$S_r$}
\put(279,143){$S_l$}
\put(138,245){$S_{ur}$}
\put(200,245){$S_{ul}$}
\put(138,45){$S_{br}$}
\put(200,45){$S_{bl}$}
\end{picture}%
\caption{The affine Riemann surface $\cal A_\infty$. The arrows indicate which side shall be glued together. See the text for a description. We used two shades of gray to illustrate Section~\ref{subsec:hlim}.}%
\label{fig:ainf}%
\end{figure}
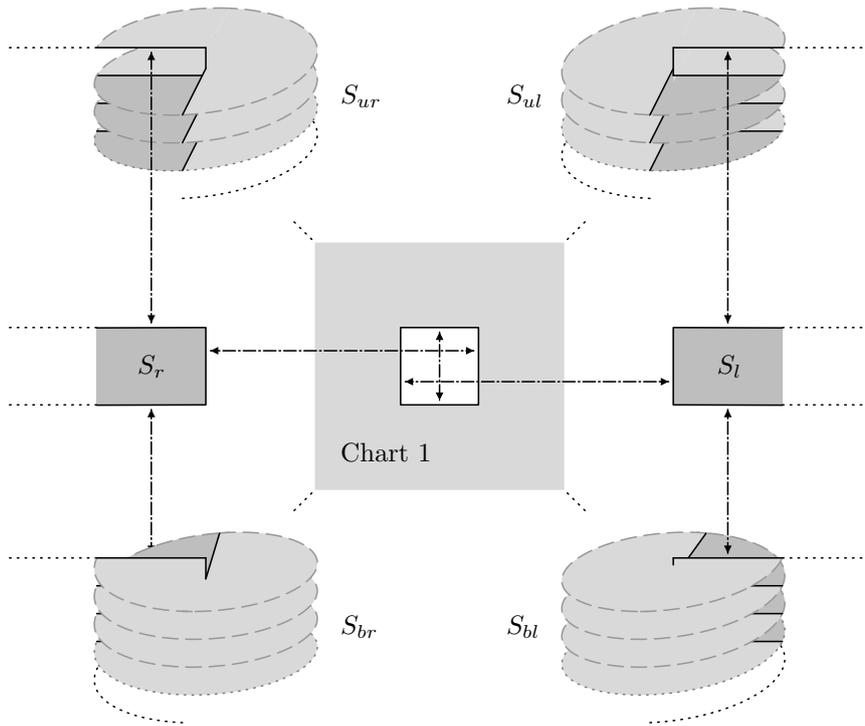

The definition of $\cal A_\infty$ was slightly more complicated. It still uses Chart~1.
Now we glue together the top and bottom segments of the boundary of Chart~1, corners excluded. On the left open segment, we glue a half infinite strip $S_l=]0,+\infty[\times]-1,1[$ by a translation. We do the same with the right segment and the strip $S_r=]-\infty,0[\times]-1,1[$. The union of the three pieces is bounded by four half lines, radiating from each corner. We will add a similar set to each of them, so let us focus on the upper left corner. Consider the universal cover of $\C^*$. Cut it in half along one lift of the half line $]0,+\infty[$. Keep the half that lies above this line and throw away the other half. In other words, we keep points with polar coordinates $(r,\theta)$ with $\theta>0$ (note that for $r\neq 0$ then $(r,\theta)\neq(r,\theta+2k\pi)$ because we work with the universal cover of $\C^*$). Glue this to the top of $S_l$ by a translation. We will call $S_{ul}$ this part. The parts $S_{bl}$, $S_{ur}$, $S_{br}$ are defined similarly ($br$ stands for bottom right, etc\ldots). See Figure~\ref{fig:ainf}.

\begin{center}\includegraphics[width=5cm]{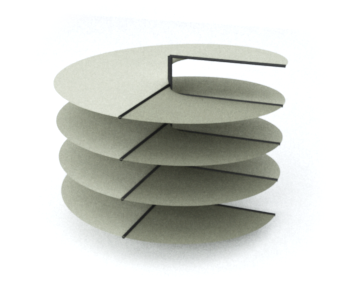}\end{center}
Just for recreation, we present above a 3-dimensional rendering of the first 4 turns of a non-conformal model of the spiral $S_{ul}$. Only a bounded part is shown: each stage should extend to infinity. The conformal structure is the pull-back by the orthogonal projection to the horizontal plane of the standard conformal structure.

\subsection{Embedding the candidate at the limit}\label{subsec:embs}

Now let us embed $\cal A_\infty$ at the limit of $\cal A_K$, as per Definition~\ref{def:1}. We begin by the charts system over 
\[ \cal A:=\coprod_{\scriptsize\begin{array}{c}K\geq 1 \text{ or}\\ K=\infty\end{array}}\cal A_K
.\]
More precisely, we will define the inverse of charts, call them $\psi$, from open subsets of $[0,1]\times \C$ to $\cal A$.
Let $x\in\cal A_\infty$.

Case 1, $x$ is in Chart~1 of $\cal A_\infty$.
Let $\psi : [0,1]\times (\C\setminus\sq) \to \cal A$ that maps $(t,z)$ to the point of $\cal A_{1/t}$ of coordinate $z$ in Chart~1.

Case 2, $x$ is on the horizontal segment. Let $x$ be represented by some $u\in\C$ in the upper segment bounding Chart~1 of $\cal A_\infty$. For any $K$, including $\infty$, let Chart 1', with image the infinite strip $]-1,1[\times \R$, be defined as follows: $U$ is the union of part of Chart~1 with real part between $-1$ and $1$ (two pieces) plus the (one or two) horizontal segments, plus if $K\neq \infty$, the rectangle. These parts are mapped to $\C$ so as to glue according to the structure of $\cal A_K$: in the upper part , map a point of coordinate $z$ to $z$; in the lower part, map it to $z+2i-2i/K$. In the rectangle, map it to $z-i/K+i$, so that the boundaries fit. This extends to the segment in a coherent way. 
We then define $\psi : [0,1]\times (]-1,1[\times \R) \to \cal A$ mapping $(t,z)$ to the point of $\cal A_{1/t}$ of coordinate $z$ in Chart~1'.

Case 3, $x$ is on an infinite half strip or its boundary. For instance $x\in S_l$.
Let $\psi$ be defined on the set of $(t,z)\in[0,1]\times \C$ such that, denoting $K=1/t$, either ($\Im(z)\in\,]-1,1[$ and $\Re(z) \leq 2K$) or $\Re(z)\in\,]0,2K[$.
It maps a point with $\Re(z)\leq 0$ to the point of coordinate $z-1$ in the closure of Chart~1, a point with $\Re(z)\geq 0$ and $\Im(z)\in [-1,1]$ to the point of coordinates $-1+z/K$ in the closure of Chart~2, points with $\Im(z)\geq 1$ to the point of coordinate $(-1+i)+(z-i)/K$ in Chart~1 and points with $\Im(z)\leq -1$ to $(-1-i)+(z+i)/K$

Case 4, $x\in S_{ul}$ (the cases of $S_{bl}$, $S_{ur}$ and $S_{br}$ are treated similarly). The set $S_{ul}$ is bijectively parameterized by polar coordinates $(r,\theta)$ with $\theta>0$. We first set up a map $\psi_0$ from a subset of $[0,1]\times S_{ul}$ to $\cal A$ that is \emph{not} an inverse of chart, as follows.
There are two cases: $\theta\in [0,3\pi/2]+2n\pi$ ($n\geq 0$) and $\theta\in [-\pi/2,0]+2n\pi$ ($n\geq 1$). Let $K=1/t$ and $z=re^{i\theta}$.
In the first case, we map $(t,(r,\theta))$ to the point of coordinates $(1+i)+z/K^{n+1} $ in Chart~1. In the second case, we define the map only if $|z|<2K^{n}$ and 
we map $(t,(r,\theta))$ to the point of coordinate $(1+i/K)+z/K^{n+1}$ in Chart~2. It does not matter if the domain of $\psi_0$ is open or not. Now for any $x\in S_{ul}$, consider the projection $\pi:S_{ul}\to\C^*$. Let $z=\pi(x)$ and choose $r>0$ such that $B(z,r)\subset\C^*$ and such that $\pi$ has an inverse branch $\xi$ on $B(z,r)$ that maps $z$ to $x$. Call $B=\xi(B(z,r))$. There is a $t_0>0$, that depends on $x$ and the choice of $r$, such that $\psi_0$ is defined on $[0,t_0]\times B$. Let then $\psi : [0,t_0[\times B(z,r) \to \cal A$ map $(t,w)$ to $\psi_0(t,\xi(w))$. See Figure~\ref{fig:case4}.

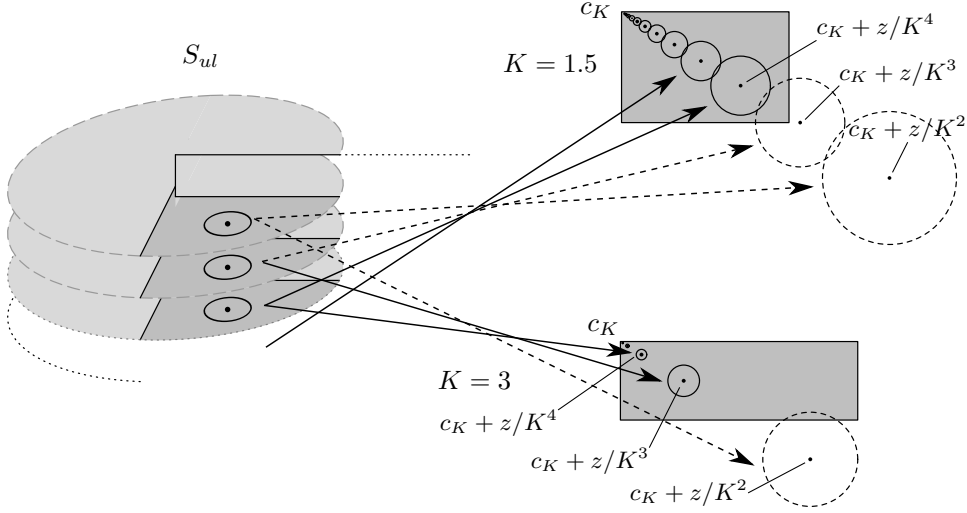
\begin{figure}
\begin{tikzpicture}[inner sep=0pt, outer sep=0pt]
\node at (0,0) {\includegraphics[width=12.5cm]{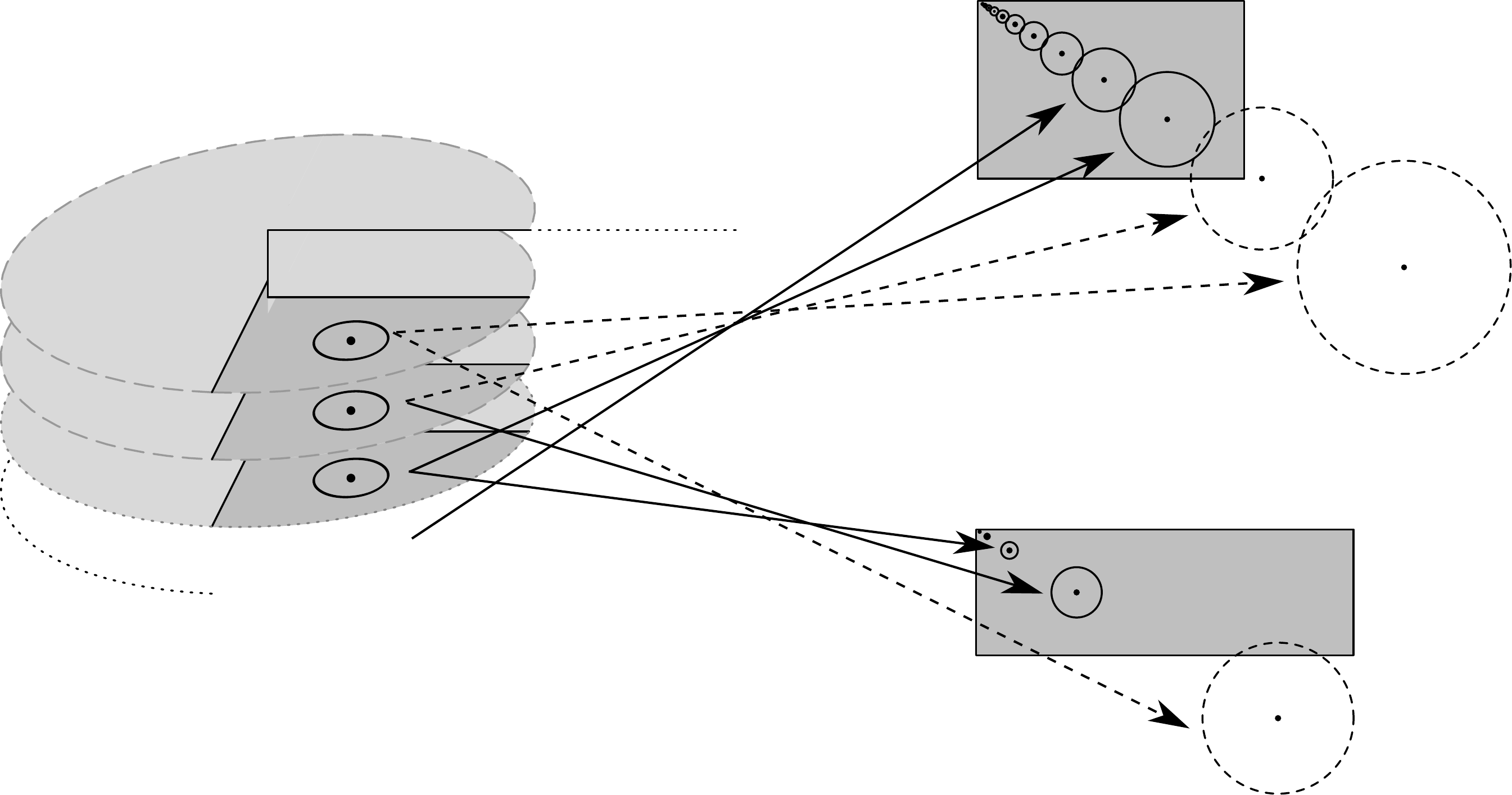}};
\node at (-3.7,2.7) {$S_{ul}$};
\node at (1.5,3.3) {$c_K$};
\node at (1.6,-0.9) {$c_K$};
\draw[-] (5.6,1.55) node[above] {\small $c_K+z/K^{2}$} -- (5.4,1.15);
\draw[-] (4.7,2.3) node[above right] {\small $c_K+z/K^{3}$} -- (4.25,1.85);
\draw[-] (4.4,2.9) node[above right] {\small $c_K+z/K^{4}$} -- (3.48,2.35);
\draw[-] (3.5,-3.1) node[left] {\small $c_K+z/K^{2}$} -- (4.25,-2.7);
\draw[-] (2.2,-2.5) node[below left] {\small $c_K+z/K^{3}$} -- (2.6,-1.7);
\draw[-] (1,-2) node[below left] {\small $c_K+z/K^{4}$} -- (2.0,-1.3);
\node at (0.9,2.6) {$K=1.5$};
\node at (-0.1,-1.6) {$K=3$};
\end{tikzpicture}
\caption{Illustration of the embedding, Case 4, for $\theta\in ]-\pi/2,0[ +2\pi n$ and $n\geq 1$. The upper left corner of the rectangle is $c_K=-1+i/K$. We not only indicated the image of $z$ but also the image of a ball around $z$, so as to illustrate separation.}
\label{fig:case4}
\end{figure}

The charts are given by the inverse of the collection of maps $\psi$ defined above. They are leafwise compatible with the original atlas of $\cal A_K$.
Let $\cal U$ be the union of their domains. It turns out that $\cal U=\cal A$ (the first three cases already cover all but part of $\cal A_\infty$, the rest is covered by Case~4), but that is not important. Put on $\cal U$ the topology generated by the charts, i.e.\ a set is open iff its image by all charts are open.

We leave it to the reader to check that the coordinate changes are continuous. Hence the topology on $\cal U$ is also the topology generated by all the preimages of open sets by charts. 

There remains to check that the topology is separated. This is a key point. It is straightforward if not both points are on $\cal A_\infty$. Otherwise, let us treat one case, the others being analogous. Assume $x\in S_l$ and $x'\in S_{ul}$ has argument $\theta\in [-\pi/2,0]+2n\pi$ ($n\geq 1$). To them we associated points $z\in\C$ and $z'\in\C$ and inverse charts $\psi$ and $\psi'$. 
When $t>0$, neighborhoods $B(z,r)$ and $B(z',r')$ are both mapped to Chart~2 by the respective inverse charts. But for $t$ sufficiently small, their images are disjoint.  Indeed, $B(z,r)$ is mapped to a ball whose center distance to the corner $-1+i$ and whose radius are both of order $1/K=t$, whereas whereas $B(z',r')$ is mapped to a ball
for which these quantities are of order $1/K^{n+1}$. Figure~\ref{fig:case4} illustrates another case: distinct points in $S_{ul}$ with the same projection on $\C$.

\subsection{Conformal maps, compactness, convergent subsequences}\label{subsec:cccs}

Now we recall that we proved in~\cite{C} that the objects $\cal A_K$, seen as Riemann surfaces, are isomorphic to the sphere minus five points.\footnote{This was done by adding charts for each hole} More precisely, for all $K$ there is a unique point $z_1=z_1(K)\in\C$ with $\Re(z)>0$ and $\Im(z)>0$ and a unique conformal map $\rho=\rho_K: \cal A_K \to \C\setminus\{z_1,\ov z_1, -z_1, -\ov z_1\}$, it was called $\phi_K$ in~\cite{C}, such that the following expansion holds at infinity in Chart~1:
$\rho(z)\underset{\infty}=z+0+\cal O(1/z)$. A more correct expression is $\sigma(z)\underset{\infty}=z+0+\cal O(1/z)$ where $\sigma=\sigma_K=\rho\circ \iota(z)$ with $\iota=\iota_K:\C\setminus\sq \to \cal A_K$ being the inverse of Chart~1. Let $z_2=-\ov{z_1}$, $z_3=-z_1$ and $z_4=z_1$. This is a counter clockwise numbering.

By the theory of univalent maps, the set of univalent maps $\C\setminus\sq \to \C$ that satisfy the normalization given by the expansion above is compact for the topology of local uniform convergence.   There is a uniform $R>0$ such that for all such map, the complement of its image is contained in $\ov B(0,R)$.
We can thus extract from all sequence $K_n\to+\infty$ a subsequence such that $\sigma_{K_n}$ converges on every compact subset of $\C\setminus\sq$ and such that moreover $z_1(K_n)$ converges to some finite complex number. By symmetry, all other $z_i$ converge.
Let us call $z_i'$ their limit and
\[\C'=\C\setminus\{z_1',\ldots z_4'\}
.\]
We will prove later that $\{z_1',\ldots z_4'\}$ consists in two distinct points: $\{-x_\infty,x_\infty\}$ with $x_\infty>0$.

\subsection{Building a limit map}\label{subsec:balm}

Thanks to the embedding, we can define the set of points of $\cal A_\infty$ that have a neighborhood on which the conformal maps $\rho_{K_n}:\cal A_{K_n}\to\C$ converge: this is the set
$\cal C$ of $x\in \cal A_\infty$ such that the following holds for any chart $\phi$ of the embedding defined on a set containing $x$ (this is independent of the choice of the chart). Let $z_0\in\C$ be such that $\phi(x)=(0,z_0)$ and let $f_n(z)=\rho_{K_n}(\phi^{-1}(1/K_n,z))$. It is a sequence of holomorphic functions defined in a neighborhood of $z_0$.
We define $x\in\cal C$ whenever this sequence to converges uniformly in a neighborhood of $z_0$. The limit at $z_0$ depends only on $x$, because a uniform limit of continuous functions is continuous in the pair: (parameter,variable).

It is then immediate that $\cal C$ is an open subset of $\cal A_\infty$.
Let us prove that it is closed, hence $\cal C=\cal A_\infty$.
Let $x\in\ov{\cal C}$ and consider the objects $z_0$, $\phi$, $f_n$ as above. 
Take $\epsilon>0$ such that for all $n$ big enough $B(z_0,\epsilon)\subset\on{Dom}(f_n)$.  
There are points $z'$ arbitrarily close to $z_0$ such that $f_n$ converges uniformly in a neighborhood of  $z'$. In particular we can take $z'\in B(z_0,\epsilon)$.
Now, $f_n$ is a sequence of univalent maps on $B(z_0,\epsilon)$. Convergence in a neighborhood of $z'$ implies convergence everywhere on $B(z_0,\epsilon)$.

Thus $\cal C=\cal A_\infty$ and there is a limit map $\rho_\infty:\cal A_\infty\to \C$. It is holomorphic since it is equal in charts to a uniform limit of holomorphic functions. The expression of its restriction to Chart~1, is the normalized injective map $\lim \sigma_{K_n}$. In particular, $\rho_\infty$ is non-constant thus, since $\cal A_\infty$ is connected, $\rho_\infty$ is nowhere locally constant.

Now recall the following theorem of complex analysis: if a locally convergent sequence of holomorphic maps $f_n:U\to X$, $U$ and $X$ being Riemann surfaces, $U$ connected, avoids $a_n$ and $a_n\tend a$, then $\lim f_n$ is either constant equal to $a$ or avoids $a$.

As a consequence, $\rho_\infty(\cal A_\infty)\subset \C'$, as can be seen by working in  charts of the embedding.

\subsection{Uniqueness of the limit}\label{subsec:meths}

We give two approaches to the problem.

\subsubsection{Method~1: Riemann-completeness}

Let us show that the holes of $\cal A_\infty$ are punctures i.e.\ one can complete it into a Riemann surface $\wh{\cal A}_\infty$ by adding three points, let us call them $\{\infty,x_-,x_+\}$, and charts near these points (this is not the only method). The hole at infinity can be covered by adding a point and the inverse chart $B(0,1/2)\to \wh{\cal A}_\infty : 0 \mapsto \infty$ and $z(\neq 0) \mapsto 1/z$.

\begin{figure}
\begin{tikzpicture}
\node at (0,0) {\includegraphics{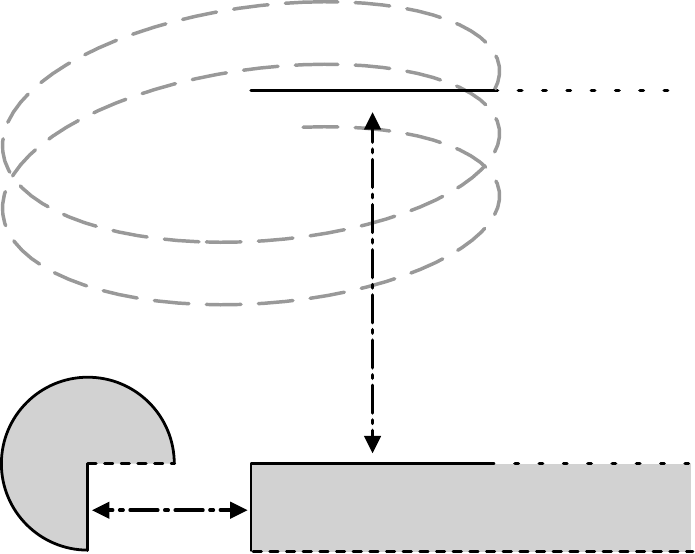}};
\node at (-2.65,-0.75) {in Chart~1};
\node at (1.2,-2.4) {half of $S_l$};
\node at (-1,1.2) {$S_{ul}$};
\end{tikzpicture}
\caption{The set A. The two densely dashed lines indicate where $A$ is glued to its symmetric $B$.}
\label{fig:A}
\end{figure}
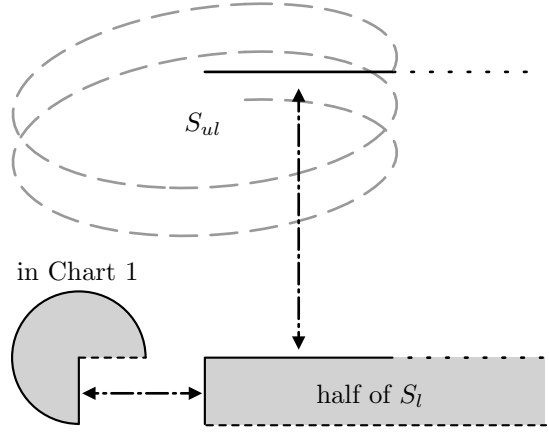

The charts at $x_-$ and $x_+$ are less explicit. Let us cut a piece off $\cal A_\infty$ as follows, so as to retain only a neighborhood of the left hole on which it will be easier to work. We remove from Chart~1 the points that are at euclidean distance $\geq 1$ from the upper left and the lower left corners. 
Recall that $1$ is half the size of the side of the square. We also remove the half strip $S_r$ attached to the right side, and the two spirals $S_{ur}$ and $S_{br}$ attached to this strip. Then we cut the rest in two halves $A$ and $B$: the set $A$ consists in points that are either in the upper spiral, or the upper half plane in $\cal S_l$ and in Chart~1. The set $B$ consists in the other points. These two pieces are glued along the real axis in $\cal S_l$ and along a horizontal radius in each trucated disk in Chart~1. Piece $A$ is illustrated on Figure~\ref{fig:A}.

\begin{figure}
\begin{tikzpicture}
\node at (0,0) {\includegraphics[width=8cm]{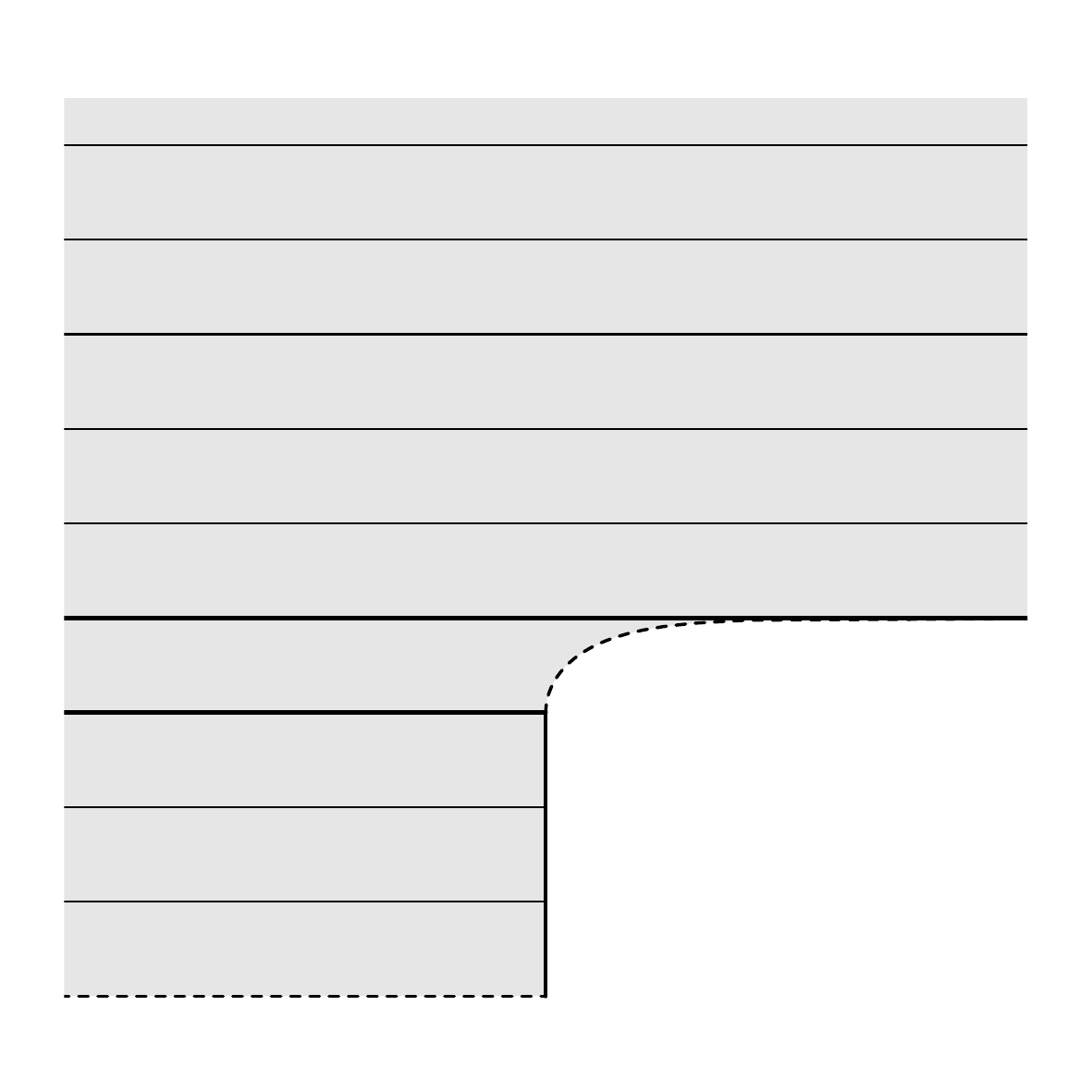}};
\node at (-1.7,-2.3) {part of Chart~1};
\node at (-1.7,-0.87) {half of $S_l$};
\node at (0,1.5) {$S_{ul}$};
\end{tikzpicture}
\caption{The image of $A$ by the logarithm. Horizontal lines correspond to heights that are multiples of $\pi/2$. Dashed lines are where $A$ and $B$ are be glued. The three pieces are bounded by the dashed lines and the thick lines.}
\label{fig:embed}
\end{figure}
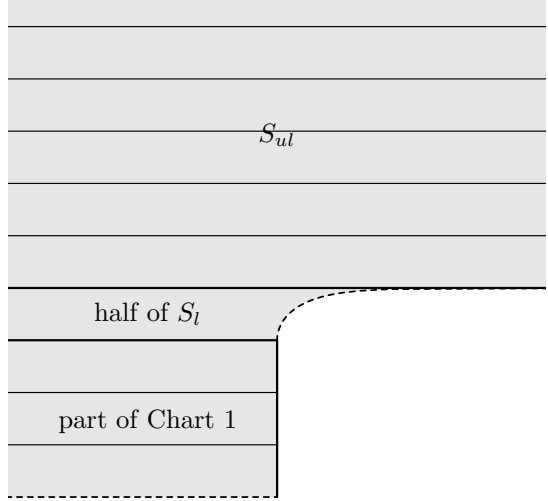

Piece $A$ naturally embeds in the universal cover of $\C^*$. Let us take its image by the logarithm. We obtain Figure~\ref{fig:embed}. The image $A'\subset\C$ is bounded by the horizontal half line whose rightmost end is $0$, the segment $[0,\frac{3\pi}2 i]$ and a curve starting from the last point of this segment, with real part tending to $+\infty$ and an asymptote equal to the horizontal line $\Im z=2\pi$.
Piece $B$ is isomorphic to the image $B'=s(A')$ by the symmetry of real axis: $s(z)=\ov z$.

\begin{figure}
\begin{tikzpicture}[>=angle 90]
\node at (0,1.5) {\includegraphics[width=4cm]{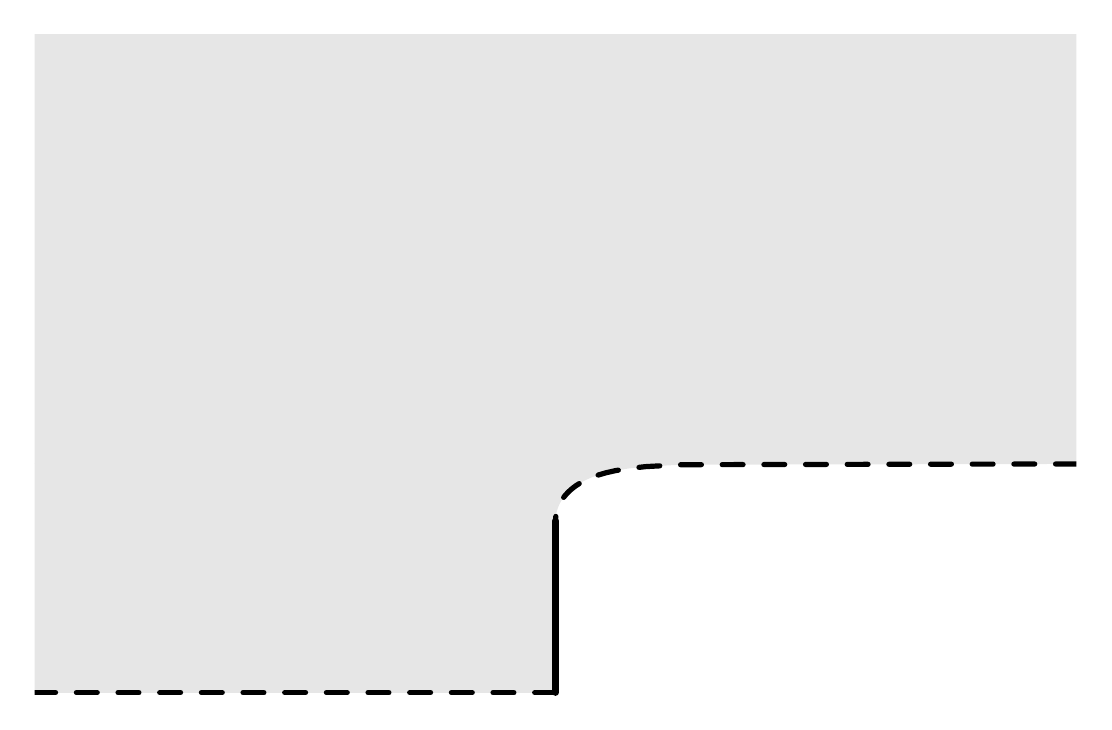}};
\node at (0,-1.5) {\scalebox{1}[-1]{\includegraphics[width=4cm]{dessin_tw_2.pdf}}};
\node at (5,0) {\includegraphics[width=4cm]{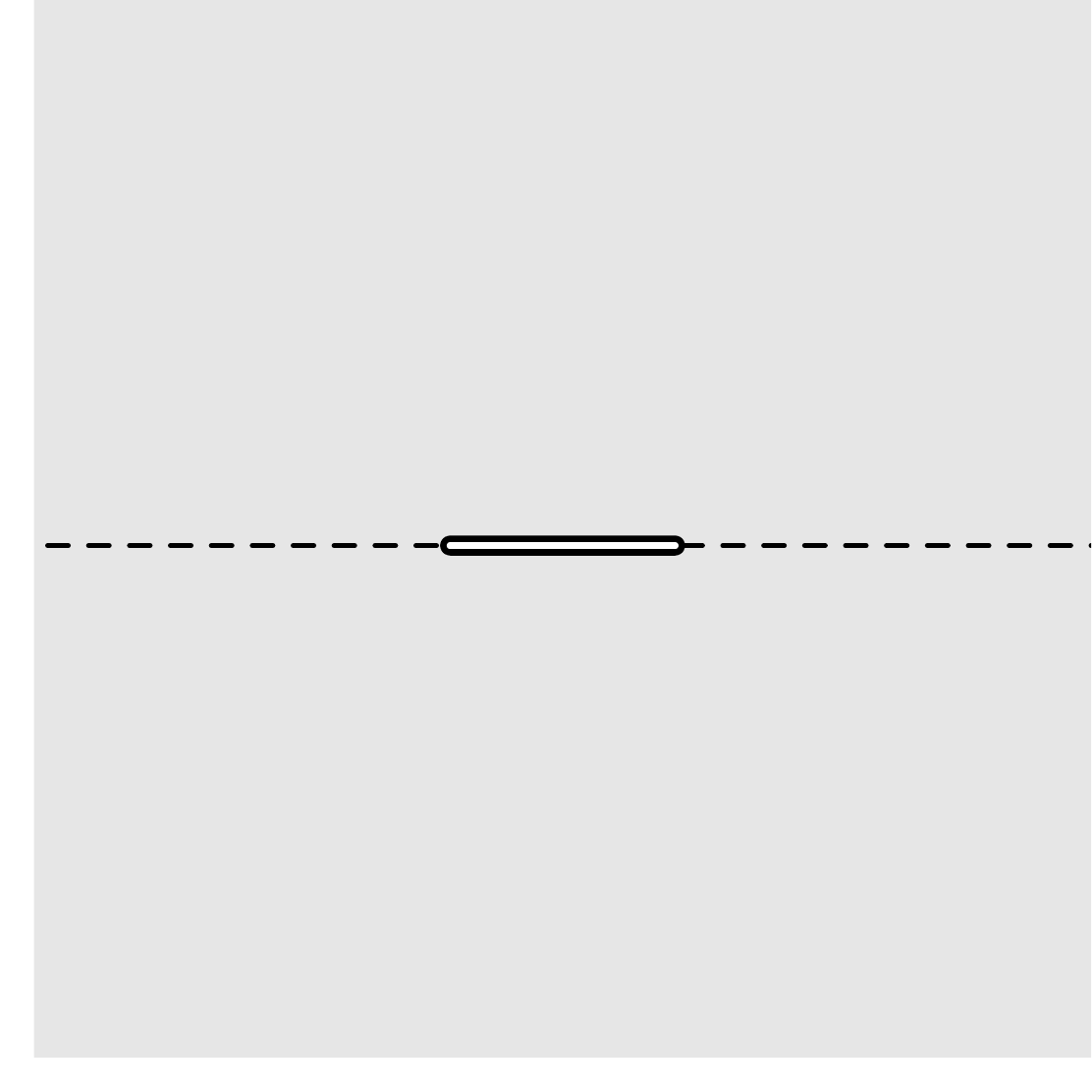}};
\node at (-1.2,1.8) {$A'$};
\node at (-1.2,-1.8) {$B'$};
\node at (5,1.5) {$\H$};
\node at (5,-1.5) {$s(\H)$};
\draw[->] (1.7,1.7) -- (3.5,1.2);
\draw[->] (1.7,-1.7) -- (3.5,-1.2);
\end{tikzpicture}
\caption{Uniformizing to half planes.}
\label{fig:toH}
\end{figure}

The gluing between $A$ and $B$ becomes a gluing between $A'$ and $B'$ along their boundaries by $z\mapsto s(z)$, except that we \emph{do not} glue the segment $[0,\frac{3\pi}2 i]$ and its symmetric. Now $A$ is bounded by a Jordan curve in $\wh{\C}$, hence the confomal maps from $A$ to the upper half plane extend continuoulsy to this boundary into a homeomorphism of the closures. The same holds for $B$ and $B'$ by symmetry.
By Schwarz reflection, joining these two isomorphisms extends to an isomorphism of the full piece to $\C\setminus [a,b]$ where $[a,b]\subset\R$, see Figure~\ref{fig:toH}. We can thus add a point at infinity in this chart. The situation for $x_+$ is similar.

We have thus completed $\cal A_\infty$ into a Riemann surface homeomorphic to $\wh{\C}$. It is thus isomorphic to $\wh{\C}$ and conformal maps from $\cal A_\infty$ to subsets of $\C$ extend holomorphically to isomorphisms between $\wh{\cal A}_\infty$ and $\wh{\C}$.
By rigidity, there is thus a unique conformal map from $\cal A_\infty$ to an open subset of $\C$ with expansion in Chart~1: $\phi(z)=z+0+\cal O(1/z)$ when $z\to\infty$.

This uniqueness implies that the limit map $\rho_\infty$ is \emph{independent} of the sequence $K_n$.

\begin{remark} 
We have extended $\cal A_\infty$ into a compact Riemann surface $\wh{\cal A}_\infty = \cal A_\infty \cup\{x_-,x_+,\infty\}$ and we have proved convergence, when $K\to+\infty$, of $\rho_K$ to an isomorphism $\rho_\infty : \cal A_\infty \to\C\setminus\{z_-,z_+\}$, where $z_-$ and $z_+$ are the images of $x_-$ and $x_+$ by an extended isomorphism $\wh{\cal A}_\infty \to \wh\C$. We can also define as in \cite{C} completions $\wh{\cal A}_K$, but it is not obvious a priori to embed $\wh{\cal A}_\infty$ at the limit of these completions, near the points $x_-$ and $x_+$. However the isomorphisms toward $\wh\C$ and the convergence of $\rho_K$ to $\rho_\infty$ give a ``magic'' way to do it.
\end{remark}

\subsubsection{Method~2: Affine-completeness}\label{subsub:meth2}

The second approach consists in first proving that there is no way to extend $\cal A_\infty$ as a connected affine Riemann surface.

\begin{figure}[h]
\begin{tikzpicture}
\node at (0,0) {\includegraphics[width=12cm]{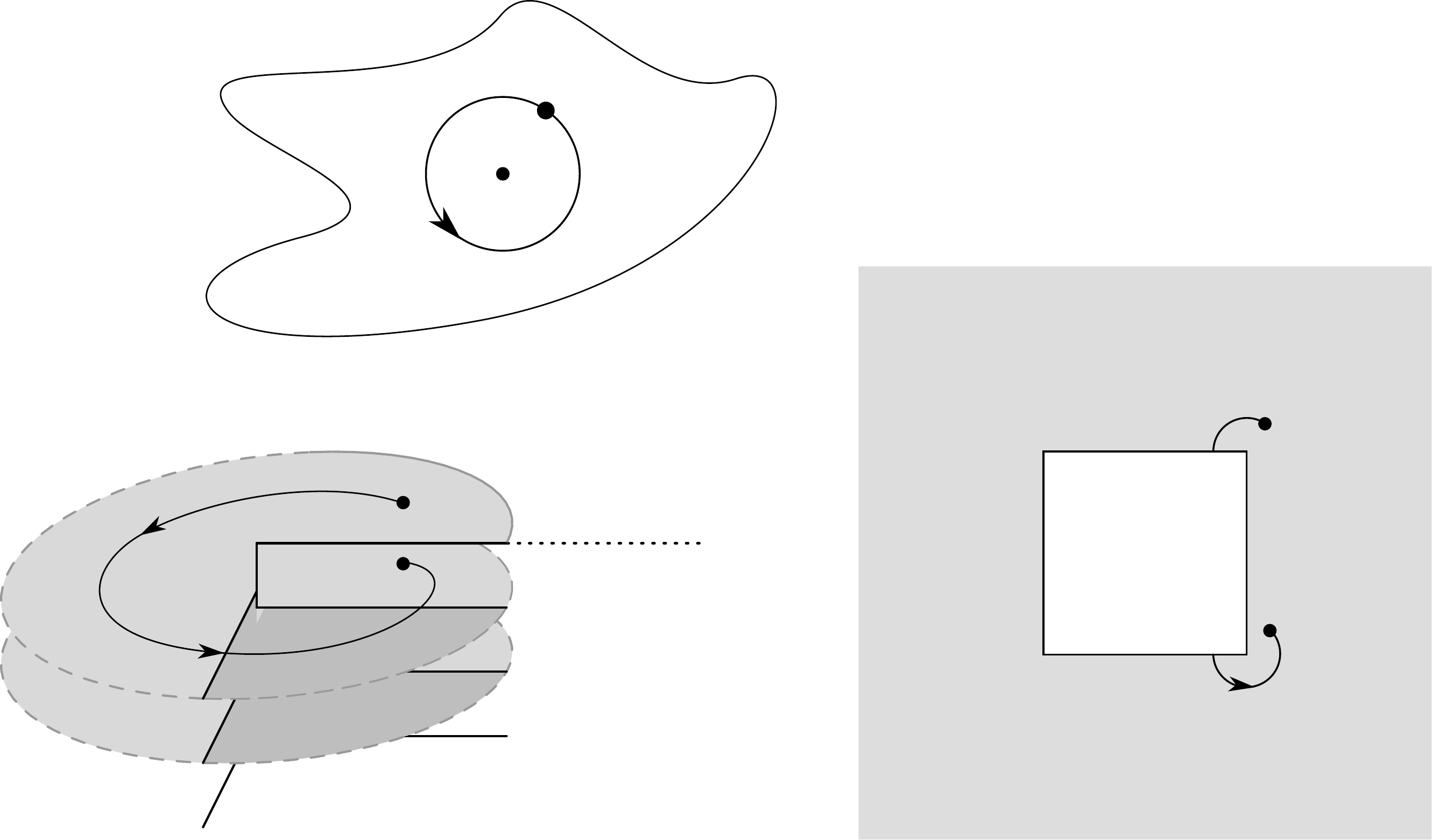}};
\node at (-2.5,1) {chart of $\cal A'_\infty$ near $x_2$};
\node at (-1.4,2) {$x_2$};
\node at (-0.1,-2.3) {basic charts};
\node at (-0.05,-2.75) {for $\cal A_\infty$};
\node at (0,-.7) {or};
\end{tikzpicture}
\caption{If there were an affine chart of an extension, centered on corner, then there would be a small loop that closes in this chart but does not in the basic charts, leading to a contradiction. Here we illustrated only two cases of non-closure, depending on the starting point.}
\label{fig:loop}
\end{figure}
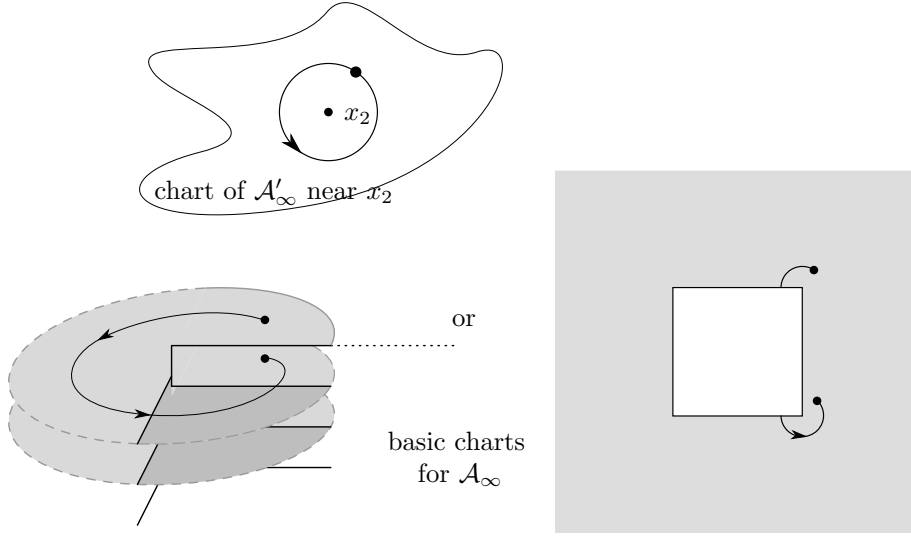

Assume the opposite and call $\cal A_\infty'$ such an extension. 
In affine Riemann surfaces there is a notion of straight lines: these are the parameterized curves that are locally pieces of lines in affine coordinates. Better: given an initial point and speed, there is a unique parameterized curve $\gamma: [0,t_{\on{max}}[$, $t_{\on{max}}\in]0,+\infty]$, such that the speed is constant in affine charts. Let us call them \emph{geodesics}.

A simple study of $\cal A_\infty$, left to the reader, reveals that the only straight lines that do not run forever, i.e.\ for which $t_{\on{max}}<+\infty$, are those that hit a corner while being eventually contained in one of the basic charts we used to define $\cal A_\infty$, or in the closure of one. 

Consider now a point $x_0$ in the boundary of $\cal A_\infty$ in $\cal A_\infty'$. Consider an affine chart near $x_0$. Let $\gamma$ be a path that is a straight segment in the chart, starting at some point $x_1\in \cal A_\infty$. 
Let $x_2=\gamma(t_2)\in \cal A_\infty'$ be the first point where the path leaves $\cal A_\infty$ when we start from $x_1$ (the point $x_2$ may or may not be equal to $x_0$).
Since the latter geodesic does not run forever in $\cal A_\infty$, it hits a corner in the way described in the previous paragraph.

In a chart for $\cal A'_\infty$ at $x_2$, we can define a small circular loop around $x_2$ starting from a point on the geodesic. It goes back to its initial point. Let us follow this path in the charts defining $\cal A_\infty$. Since the change of coordinates are affine, this path is circular in each basin chart, centered on right corners or on left corners, and a close study in basic charts shows that
this its endpoint and starting point will be different in all cases (for instance, a point on another sheet of the spiral attached to the corner), see Figure~\ref{fig:loop}. This is a contradiction. Q.E.D.

Second, we use the particular form of the connection. 
Let us recall that if we transport the affine atlas of $\cal A_K$ by the conformal map $\rho_K:\cal A_K \to \C\setminus\{z_1(K),\ldots,z_4(K)\}$, the affine charts $\phi : U\to \C$ (where $U\subset\C\setminus\{z_1(K),\ldots,z_4(K)\}$) satisfy that the quantity $\zeta(z):=\phi''(z)/\phi'(z)$ depends on $z$ but not on $\phi$. We proved in~\cite{C} that the function $\zeta$ has the following form:
\[\zeta_K(z)=\frac{\log K}{2\pi i} \left(\frac{-1}{z-z_1}+\frac{1}{z-z_2}+\frac{-1}{z-z_3}+\frac{1}{z-z_4}\right).\]
The affine charts on $\cal A_\infty$ also give rise to a function $\zeta_\infty$, defined on the subset $\rho_\infty(\cal A_\infty)$ of $\C'$ (recall that $\C'$ is defined as the complement of the limit of the $z_i$).

Let us prove that $\rho_\infty(\cal A_\infty)= \C'$. From the convergence of $\rho_{K_n}$ to $\rho_\infty$ (w.r.t.\ the embedding of $\cal A_\infty$ at the limit of $\cal A_{K_n}$), we deduce that $\zeta_{K_n}\tend \zeta_\infty$ on every compact subset of $\rho_\infty(\cal A_\infty)$. 
Let us write $\zeta_K = P_K/Q_K$ with $Q_K(z)=(z-z_1)\cdots(z-z_4)$. Since $z_{K_n}$ converges, the denominator $Q_K$ converges.
Now since $\zeta_{K_n}$ converges on an open set, the sequence of polynomial $P_{K_n}$, of bounded degree, converges on an open set, hence its coefficients converge, so it converges uniformly on compact subsets of $\C$. Hence $\zeta_{K_n}$ converges uniformly on every compact subsets of $\C'$ to a rational map $f$ whose poles are contained in $\C'\setminus\C$, i.e.\ the limits of the $z_i$. Hence $\zeta_\infty$ is the restriction of $f$ to $\rho_\infty(\cal A_\infty)\subset \C'$. If the latter inclusion was strict, one could use $f$ to extend $\cal A_\infty$ as an affine Riemann surface. But we saw that this was impossible. Q.E.D.

Now that we know that $\cal A_\infty$ is isomorphic, as a Riemann surface, to the Riemann sphere minus a finite number of point, it follows that $\rho_\infty$ is independent of $K_n$, because we normalized it (by imposing $\rho_\infty(x) = z +0+\cal O(1/z)$ when $z\tend\infty$, where $z$ is the coordinate of $x$ in Chart~1).

\subsection{Convergence}\label{subsec:convergence}

Uniqueness of the limit, that we just proved, implies convergence, when $K\to+\infty$, of $\rho_K$ to an isomorphism $\rho_\infty : \cal A_\infty \to \rho_\infty(\cal A_\infty) \subset\C$. We have not yet proved convergence of the $z_i$ as $K\to +\infty$. This will also be proved by uniqueness of their limits.

Recall that there is some $R>0$ independent of $K$ so that all $z_i$ belong to $\ov B(0,R)$, as explained in Section~\ref{subsec:cccs}. Recall that for all sequence $K_n\tend +\infty$, we can extract a subsequence such that the $z_i$ converge.

We know, as explained at the end of Section~\ref{subsec:balm}, that for such a sequence $K_n$, we have $\rho_\infty(\cal A_\infty)\subset \C'$ where $\C'$ is $\C$ minus the limit of the $z_i$. We also proved that the set $\rho_\infty(\cal A_\infty)$ is independent of the choice of $K_n$. We obtained by two different methods in Section~\ref{subsec:meths} that $\rho_\infty(\cal A_\infty) = \C \setminus \{\text{a finite set}\}$. With the first method, we proved that this set has two elements. With the second, we have proved that $\rho(\cal A_\infty)=\C'$.

One way to deduce convergence of the $z_i$ from either method is to first prove that the $z_i$ vary continuously when $K<+\infty$ and then since their cluster set is finite, they have to converge.

Let us give two alternative proofs. We will also identify which $z_i$ have a common limit, and obtain information on the position of the limits in the plane. These proofs will be formulated so that either of the two methods of Section~\ref{subsec:meths} can be used as a starting point, without needing to use both.

\subsubsection{}

This approach is probably the most general and powerful, but a bit long to detail. First see that $\cal A_\infty$ has three holes, as a topological surface (of course if one comes from Method~1, then one already knows that and better). One gets that one of them is $\infty$, one corresponds to the two left corners and one to the two right corners. Consider the hole on the ``left''. Take a small loop around it, call it $\gamma_\infty$. Then if one is able to define parameterized loops $\gamma_K$ that tend to $\gamma_\infty$, and see that it separates from $\infty$ the two left corners, this will prove that $z_2(K)$ and $z_3(K)$ are separated from $\infty$ by the image of $\gamma_K$ by $\rho_K$, which tends to $\rho_\infty(\gamma_\infty)$. By choosing $\gamma_\infty$, we can ensure that the latter is small  because the holes are punctures (we gave a direct proof of that in Method~1; in Method~2 it follows from the fact that $\cal A_\infty$ is isomorphic to the sphere minus finitely many points). So we know that there are two distinct points $z_-$ and $z_+$ so that $z_2(K)$ and $z_3(K)$ tend to $z_-$ and so that $z_1(K)$ and $z_2(K)$ tend to $z_+$.

There is an orientation reversing symmetry of $\cal A_\infty$ that preserves the two holes. This symmetry is conjugated by $\rho_\infty$ to an anti-holomorphic symmetry of $\C$ minus a finite number of points. This symmetry automatically extends to $\wh\C$ by rigidity, and according to the normalization\footnote{$\rho_\infty(z)=z+0+\cal O(1/z)$ at $\infty$, expressed in Chart~1}, it is conjugated to $z\mapsto \ov z$. Hence the points $z_-$ and $z_+$ belong to $\R$. There is another symmetry, with vertical axis. This symmetry is conjugated to the reflection by the imaginary axis in $\C$. Hence $z_-=-z_+ \neq 0$. Interestingly, this implies that the axis of reflection is mapped onto the imaginary axis. To see that $z_- <0$, use for instance the fact that the left hole is on the left of the symmetry axis, the latter being the intersection of the imaginary line with the closure of Chart~1. According to the normalization at $\infty$, the map $\rho$ from this axis to the imaginary axis respects the natural orientation of these lines (i.e.\ upward is mapped to upward). Nearby points on the left of one axis are thus mapped to nearby points on the left of its image. The left component is thus mapped on the left.

\subsubsection{}
Here, we start from the symmetries of $\cal A_\infty$. By similar arguments, using rigidity and the normalization at $\infty$, the symmetry with vertical axis is conjugated to the reflection by the imaginary axis.
The normalization at infinity also implies that the image of the vertical axis is the whole imaginary axis, by the intermediate value theorem.
On the other hand, there must be a merge between the limits of the $z_i(K_n)$.
Indeed in Method~1 we have that $\C\setminus \C'$ has at most two points. And in Method~2, we can use the fact that the residue at $z_i$ is $\pm\log K/2\pi i$ and tends to infinity: without a merge, the functions $\zeta_{K_n}$ would tend to $\infty$ everywhere on $\C'$. Recall that $z_2=-\ov{z_1}$, $z_3=-z_1$, $z_4=\ov{z_1}$.
Because of these symmetries between the $z_i$, and because the image of $\cal A_\infty$ is disjoint from $\C'$ and contains the imaginary axis, the only possible merge is that $z_1$ and $z_4$ tend to a positive limit $u_+$, and that $z_2$ and $z_3$ tend to the opposite $u_-=-u_+$: in particular $\C\setminus\C'=\{u_-,u_+\}$. 
In the Method~1, since $\C\setminus\C' = \{u_-,u_+\} \subset\C\setminus\rho_\infty(\cal A_\infty)$ and since we already know that the latter has exactly two points $z_-$ and $z_+$, we conclude that $\rho_\infty(\cal A_\infty)=\C'$ and that $\{z_-,z_+\}=\{u_-,u_+\}$. To see that $z_-=u_-$ and $z_+=u_+$, and not the other possibility, we can use separation by the vertical axis, as in the previous paragraph. In Method~2, we proved that $\rho_\infty(\cal A_\infty)=\C'$.

\begin{remark}
We obtain thus for Method~2 a proof that $\cal A_\infty$ is homeomorphic to the sphere with three holes. This was proved at an early stage in Method~1. The two proofs are independent.
\end{remark}

\subsection{Limit of the rectangle}\label{subsec:hlim}

Let $W_K$ be the image by the isomorphism $\rho_K : \cal A_K\to \C\setminus\{z_1(K),\ldots,z_4(K)\}$ of the open rectangle defined by Chart~2. Let $C_K=\ov{W_K}$. 

We can now prove that there is a limit, in the sense of Hausdorff, of $C_K$, and determine this limit.
Let us prove that it is equal to the union $C_\infty = \{z_-,z_+\} \cup \rho_\infty(M)$ where $M_\infty$ is the union of the following subsets of $\cal A_\infty$, also illustrated on Figure~\ref{fig:ainf} by the dark gray shade: the closure of the two half strips $S_l$ and $S_r$, and for each of the $4$ spirals $S_{ul}$, $S_{bl}$, $S_{ur}$, $S_{br}$, an infinite sequence of subsets that we explicit here only for $S_{ul}$ (the other cases being analog): recall that $S_{ul}$ can be seen as a subset of the universal cover of $\C^*$. We keep only those points whose projection to $\C^*$ belong to the closed lower right quadrant.

\begin{theorem} In the sense of  Hausdorff, $C_K \tend C_\infty$ when $K\tend+\infty$.
\end{theorem}
\begin{proof}
First, as limits of corners, the points $z_-$ and $z_+$ belong to the limit. Second recall the embedding $\cal A$ of $\cal A_\infty$ at the limit of $\cal A_K$. Let $M_K$ be the closure of the rectangle in $\cal A_K$.
By inspecting the embedding, we see that given any point $x\in \cal A_\infty$, either $x\in M_\infty$ and then all neighborhoods of $x$ in $\cal A$ intersect $M_K$ when $K$ is big enough, or $x\notin M_\infty$ and then there is a neighborhood of $x$ that intersect none of the $M_K$. The conclusion follows from this and from the convergence $\rho_K\tend\rho_\infty$ w.r.t. the embedding.
\end{proof}

\section{A second look}

Here, we adopt a more abstract point of view. Let $\cal M_n$ be any sequence of manifolds of one kind (topological, differentiable, analytic, Riemann surfaces, affine Riemann surfaces, \ldots). Let us call \emph{legal maps} the valid changes of charts for this kind: they form a subset of the set of homeomorphisms between open subsets of $\R^d$.
Recall that we defined 
\[\I = \setof{1/n}{n>0}\cup\{0\} = \{1,1/2,1/3,\ldots,0\}.\]
Since $\I\subset\R$, we may consider $\I\times\R^d$ as a subset of $\R^{d+1}$. We endow $\I\times\R^d$ with the topology induced from the standard topology on $\R^{d+1}$.
This set decomposes as \emph{leaves}: $L_\infty$ and $L_n$ with $n\in\N^*$:
\[L_n=\{1/n\}\times \R^d,\quad L_\infty=\{0\}\times\R^d.\]
Let us define a legal change of coordinates $\psi$ between open subsets $U$ and $U'$ of $\I\times \R^d$ as a map that is defined on $U$, sends $U$ to $U'$, sends each leaf into itself, is legal on each leaf, and is continuous on $U$ for the topology on $\I\times \R^d$. Because legal maps on $\R^d$ are continuous, the last condition is equivalent to asking $\psi_n$ to converge uniformly to $\psi_\infty$ on compact subsets of $U\cap L_\infty$, where leaves are identified with $\R^d$ and $\psi_n$ is the restriction of $\psi$ to $U\cap L_n$.

\begin{definition}
A \emph{virtual point} associated to the sequence $\cal M_n$ is an equivalence class of the following objects:
\\
$\bullet$ a point $a\in L_\infty$ and an open subset $U$ of $\I\times\R^d$ such that $a\in U$, together with maps $\phi_n : U\cap L_n\to \cal M_n$ each of which is the inverse of a valid chart for $\cal M_n$,
\\
for the following equivalence relation:
\\
$\bullet$ $(a,U,(\phi_n)_{n>0}) \sim (a',U',(\phi'_n)_{n>0})$ whenever there exists open neighborhoods $V$ and $V'$ of $a$ and $a'$ in $\I\times \R^d$ and a legal change of coordinates $V\to V'$ that is a restriction on each leaf of $\phi'_{n}\circ \phi_n^{-1}$ (which is a homeomorphism between open subsets of $L_n$).
\end{definition}

Note that in this definition, there is no $\cal M_\infty$ and no $\phi_\infty$.
The definition of equivalence implies that there is a neighborhood of $a$ that is  contained in the domain of definition of $\phi'_{n}\circ \phi_n^{-1}$ for $n$ big enough.

\begin{example}
As an illustration, consider the sequence $\cal A_n = \cal A_{K_n}$ of affine Riemann surfaces that we defined in the previous part, with $K_n\to +\infty$. We also defined an affine Riemann surface $\cal A_\infty$. In Section~\ref{subsec:embs}, we defined charts of an embedding of $\cal A_\infty$ at the limits of $\cal A_{K_n}$, that we can use to define virtual points. For instance consider a point in the spiral $S_{ul}\subset\cal A_\infty$ attached to the upper left corner, and situated close to the common boundary line with the left half strip $S_l$. Let this point have coordinate $z_0=\rho e^{i\theta}$ in a leaf of the chart associated to $S_{ul}$: $\theta>0$ is close to $0$. Choose $r>0$ small so that the ball $B(z_0,r)$ is contained in the upper half plane. For $n>0$, let $\phi_n:B(z_0,r)\to $ Chart~1, $\phi_n(z)= (-1+i) + z/K_n$. The image of $B(z_0,r)$ by $\phi_n$ is a sequence of balls of size $r/K_n$ centered on a point at distance $|z_0|/K_n$ from the upper left corner. The virtual point is the equivalence class of $(z_0,\I\times B(z_0,r),(\phi_n))$, but we find it easier to think of it using the sequence of balls that we just mentionned.
\end{example}

Let $\frak{M}$ be the set of all virtual points associated to the sequence $\cal M_n$.
An embedding of a manifold $\cal M_\infty$ at the limit of $\cal M_n$, as per definition~\ref{def:1} induces a map $\cal M_\infty \to \frak M$. The separation axiom, in the definition of an embedding at the limit, implies that this map is injective.

\begin{example}
In Section~\ref{subsec:embs}, we proved separation of the embedding of $\cal A_\infty$ at the limit of $\cal A_{K_n}$, hence $\cal A_\infty$ injects in $\frak M$.
\end{example}

We will now put a topology on $\frak M$ via a basis of open sets. Recall that a basis for a topology on a set $X$ is a set of subsets of $X$, called \emph{base elements}, that satisfies the following two axioms:
\begin{enumerate}
\item base elements cover $X$,
\item for all points $x$ in the intersection $A$ of two base elements, there exists a base element containing $x$ and contained in $A$.
\end{enumerate}
Then we can define open sets as all unions, finite or infinite, of base elements.
The following collection of subsets $B$ of $\frak{M}$ satisfies the two axioms and is thus a basis for a topology $\cal T$ on $\frak{M}$:
given an open subset $U$ of $\I\times \R^d$ and a collection of maps $\phi_n: U\cap L_n\to \cal M_n$, let $B$ be the set of equivalence classes of $(b,U,(\phi_n)_{n>0})$ where $b$ varies in $U_\infty$. We now have a topology.

The set $B$ is open, by definition, and the map $U_\infty \to B: b\mapsto [(b,U,(\phi_n)_{n>0})]$ is a homeomorphism.
Therefore, these maps form an atlas on $\frak M$, whose change of coordinates are legal for the kind of manifold under study.

This atlas almost turns $\frak M$ it into a manifold: for this we have to drop, in the definition of a manifold, the following two requirements: Hausdorff separation axiom, and existence of a countable basis of neighborhoods (a.k.a.\ second countability). 

The space $\frak M$ will almost never be second countable.  
Concerning separation, it depends on the structure under study. For instance, for topological or for differentiable manifolds of dimension $>0$, $\frak M$ is not separated, which may seriously limit its usefulness.

\begin{proposition}\label{prop:2}
For $G$-manifolds\footnote{In the sense of \cite{Thu}} with $G$ a subgroup of the similitude group of $\R^d$, $\frak M$ is Hausdorff separated.
\end{proposition}

\begin{proposition}\label{prop:1}
For Riemann surfaces, $\frak M$ is Hausdorff separated.
\end{proposition}

Recall that second countability is automatic for connected Riemann surfaces satisfying the Hausdorff axiom.
Hence, for Riemann surfaces, every connected component of $\frak M$ is a Riemann surface! 

The proof of Proposition~\ref{prop:2} is much easier than the proof of Proposition~\ref{prop:1} and is left to the reader.\footnote{We do not claim it is trivial either, but most of its difficulty will reside in unwrapping the statement.} We included a proof of Propostion~\ref{prop:1} in Section~\ref{subsec:pfsep} 

A way to interpret Proposition~\ref{prop:1}, and our main motivation for proving it, is that, given any sequence of charts, with the assumption that their images in $\C$ contain a given common connected non empty open set $U$, we can associate to it a unique maximal limit Riemann surface, namely the component of $\frak M$ containing any (thus all) of the virtual points with basepoint in $U$ and associated to the sequence of charts.

Continuing to discuss the case of Riemann surfaces, note that each component of $\frak M$ is a well understood object: a Riemann surface, but the set of components is extremely big, as there are a lot of inequivalent sequences for virtual points, especially if we zoom on points, and inequivalent rates of zoom yield different virtual points. Classifying them is a complicated task whose outcome is maybe not so rewarding.

\medskip\noindent\textit{A word of caution:} The object $\frak M$ is not embedded at the limit of the $\cal M_n$, but it is in a weaker sense: one has to give up separation. Indeed, shall separation hold on $\frak M$ or not, separation will always fail on the union of $\frak M$ and the $\cal M_n$. Worse: even if we consider the embedding at the limit of $\cal M_n$ of only one component of $\frak M$, it may fail to be a separated embedding: an example is given by $\cal M_{2n}=\C$ and $\cal M_{2n+1}=\C/\Lambda$, where $\Lambda$ is a rank $2$ lattice. The component $X$ of $\frak M$ containing the virtual points associated to the canonical charts is isomorphic to $\C$. (We can also devise an example with all $\cal M_n=\C/\Lambda$, using charts with oscillating scales.)
Take $\cal M_\infty$ to be an open fundamental parallelogram for $\C/\Lambda$. Then with the same example as above, $\cal M_\infty$ is well-embedded, thus injects in a component $X$ of $\frak M$. However $X\simeq \C$ is only weakly embedded at the limit of $\cal M_ n$. I.e.\ one may have a good embedding in some subset, that fails to extend to a good embedding on a larger scale. Components of $\frak M$ that are well embedded at the limit of $\cal M_n$ should be considered as better behaved.

\begin{example} Let us work with affine Riemann surfaces and the associated notion of virtual points. In the case of the $\cal A_{K_n}$, the image of $\cal A_\infty$ was a whole component of $\frak M$, because $\cal A_\infty$ cannot be extended as an affine Riemann surface.
\end{example}

Let $\cal M_n$ and $\cal M'_n$ be two sequences of manifolds of the same kind, either Riemann surfaces or $G$-manifolds with $G$ a subgroup of the similitude group of $\R^d$. Let $\frak M$ and $\frak M'$ denote the respective space of virtual points. Assume that there are injective maps $\xi_n: \cal M_n \to \cal M'_n$, regular in the sense that, expressed in charts, they belong to the set of legal maps. This naturally defines a map $\Phi:\frak M \to \frak M'$ as follows: send a virtual point with representative $(a,U,(\phi_n))$, to the virtual point of representative $(a,U,(\xi_n\circ\phi_n))$. Then $\Phi$ is continuous and injective. In particular, if we let $\cal L$ be a connected component of $\frak M$, then $\xi_n$ converges to an injective holomorphic/$G$ map from $\cal L$ to the component $\cal L'$ containing $\Phi(\cal L)$.

\begin{example}
Let us now work with the notion of virtual points associated to Riemann surfaces instead of affine Riemann surface.
Then $\cal A_\infty$ also maps to a whole component of the space $\frak M$ associated to $\cal A_{K_n}$, but for another reason than in the affine case: though $\cal A_\infty$ can be completed into a Riemann sphere by adding $3$ points, the embedding at the limit cannot extend to any of these point.
Indeed let us consider the injections $\xi_n=\rho_n:\cal A_{K_n}\to \wh\C$, and the associated map $\Phi:\frak M\to \frak M'$ where $\frak M'$ is associated to the sequence $\cal M'_n=\wh\C$.
Using the canonical charts ($z$ and $1/z$) on $\wh \C=\cal M'_n$, we get a canonical component $\cal C$ of $\frak M'$ together with an isomorphism $\wh \C \to \cal C$.
Recall that $\C'$ refers to the complement of $\infty$ and of the limit of the $z_i(K_n)$. Let $\cal C'$ be the image of $\C'$ in $\cal C$.
Recall that the map from $\cal A_\infty$ to $\frak M$ injective. Let $\cal K_\infty\subset\frak M$ be the image of $\cal A_\infty$ and $\cal L$ be the component of $\frak M$ that contains $\cal K_\infty$. 
Then $\Phi(\cal L)$ has to be contained in $\cal C'$, by the same arguments as in Section~\ref{subsec:balm}. The restriction of $\Phi$ to $\cal K_\infty$ corresponds to $\rho_\infty$. 
But we saw that $\rho_\infty(\cal A_\infty)=\C'$. Since $\Phi$ is injective, this implies that $\cal L=\cal K_\infty$.
\end{example}

\subsection{Proof of separation of the space of virtual points, for Riemann surfaces}\label{subsec:pfsep}

\subsubsection{Preliminary lemmas} Let $\D$ denote the unit disk in $\C$.

\begin{lemma}\label{lem:1}
If $r<r_0=1/16$, then\footnote{The associated extremal problem seeks for a given $r\in(0,1)$ to minimize $m$ under the assumption that $0\in U$ and $B(0,r)\not\subset U$. If the extremal case is for $U=\D\setminus[r,1)$, then we can upgrade to $r_0=1/(3+\sqrt8)$.}
 for all simply connected open subsets $U$ of $\D$ that contains $0$ but does not contain $B(0,r)$, the harmonic measure $m$ as seen from $0$ of $\D\cap \partial U$ w.r.t.\ $U$, satisfies $m>1/2$.\end{lemma}
\begin{proof}
The conformal radius $r_c$ of $U$ w.r.t. $0$ satisfies $\log r_c =\int \log|z| d\mu(z)$ where $z$ varies in $\partial U$ and $\mu$ is the harmonic measure. Since on $\partial \D$, $\log|z|=0$, we get $\log r_c \geq m \log r$. Both sides being negative, $m\geq \log r_c / \log r = \log (1/r_c) / \log (1/r)$.
By Koebe's 1/4 theorem, $r_c \leq 4 r$, hence $m\geq 1 - \log 4/\log (1/r)$.  
\end{proof}

Recall that the core geodesic of an annulus of finite modulus $m$ is the image of the equator $\R/\Z$ by the conformal representation to the following bounded cylinder: $``|\Im(z)|<m/2"/\Z$.

\begin{prelemma}\label{prelem:0}
Consider an annulus $U$ contained in $\D$ and with one boundary curve being $\partial \D$. Call $\cal C$ its core curve. Let $U'=U\cup\partial \D\cup s(U)$ where $s$ is the symmetry w.r.t.\ $\partial \D$. Then any point in one of the two curves $\partial \D$ and $\cal C$ is at a distance $d_0$ from the other curve that is independent of $U$. One computes $d_0=\log(1+\sqrt 2)/2$.
\end{prelemma}
\begin{proof}
The uniformization from $U$ to the subset $\Im(z)\in\left]0,m\right[$ of $\C/\Z$ admits a Schwarz reflection that is a uniformization from $U'$ to the sub-cylinder $C_m=\Im(z)\in\left]-m,m\right[$ of $\C/\Z$.
It sends the hyperbolic metric of $U'$ to that of $C_m$ and it maps the two curves to $\Im(z)=0$ and $\Im(z)=m/2$. By a symmetry of vertical axis, the shortest path in $C_m$ from a point of one the horizontal curves to the other is a straight vertical segment.
As a cover, the natural map $\C\to\C/\Z$ preserves the hyperbolic metric. The sub-cylinder $C_m$ becomes a strip $\Im(z)\in\left]-m,m\right[$ in $\C$, and the two closed curves become straight lines $\Im(z)=0$ and $\Im(z)=m/2$.
Rescaling the picture by $1/m$, also preserves the respective hyperbolic metrics, and the picture then becomes independent of $U$ and of $m$.
\end{proof}

The value of $d_0$ is not important here.

\begin{lemma}\label{lem:2}
For all $\eta>0$ there exists\footnote{Here too, it is likely that the extremal case is the complement of a segment: $U=\D\setminus[0,\eta]$.} $\epsilon>0$ such that for all annulus contained in $\D$ with one boundary component equal to $\partial\D$, not containing $0$, and such that the closest point to $0$ of the core curve is at distance $<\epsilon$, then the bounded component of the complement of the annulus is contained in $B(0,\eta)$.
\end{lemma}
\begin{proof} (We are not trying to get optimal bounds.) 
Let $U$ be the annulus and $U'$ as in Pre~Lemma~\ref{prelem:0}. 
Consider their images $A$ and $A'$ in $\C/\Z$ by $z\mapsto \log(z)/2\pi i$.
Koebe's one quarter theorem applied to a lift by the cover $\C\to\C/\Z$ implies that the coefficient $\rho(z)$ of the hyperbolic metric $\rho(z)|dz|$ of $A'$ is $>1/4d_e(z,\C\setminus A ')$ where $d_e$ denotes the euclidean distance.
Now if there is a point at height $h>0$ and not in $A$, then all points in $A'$ are at distance at most $\sqrt{h^2+(1/2)^2}$ from $\C\setminus A'$, (points above $h$ are at distance $\leq 1/2$ because $\C\setminus A'$ is connected and $1$-periodic), thus we get a lower bound on the coefficient: $\rho(z)>\rho_0=1/4\sqrt{h^2 + 1/4}$ whose precise value is not important either. Points in the core curve of $A$ cannot be farther than $d_0/\rho_0$ from $\R/\Z$ for the euclidean distance in $\C/\Z$.
\end{proof}

As an almost immediate corollary:

\begin{corollary}\label{cor}
For all $\eta>0$ there exists $\epsilon>0$ such that for all annulus $A$ contained in $\D$ with one boundary component equal to $\partial\D$, and such that $B(0,\epsilon)$ contains a point in $\D\setminus A$ and a point in the core curve, then $\D\setminus A\subset B(0,\eta)$.
\end{corollary}

Last:

\begin{lemma}\label{lem:3}
For all $\eta>0$ there exists $\epsilon>0$ such that for all annulus $U$ contained in $\D$ with one boundary component equal to $\partial\D$, satisfying $\D\setminus U\subset \ov B(0,\epsilon)$ then the core curve of $\D$ is contained in $\ov B(0,\eta)$.
\end{lemma}
\begin{proof}
Let $V=\D\setminus\ov{B}(0,\eta^2)$. Its core curve is $\partial B(0,\eta)$.
Here it is easy to see that the extremal case is given by $U=V$. Let $U$ satisfy the asumptions with $\epsilon=\eta^2$.
Consider the annulus $U' = U \cup \partial \D\cup s(U)$ where $s$ is the symmetry w.r.t. $\partial \D$, and define $V'$ analogously. Then $V'\subset U'$ thus the hyperbolic metric of $U'$ is $\leq$ the hyperbolic metric of $V'$.
If there were a point in the core curve of $U'$ and not belonging to $\ov B(0,\eta)$, then the straight ray from that point to $\partial \D$ would be shorter for the hyperbolic metric in $U'$, than the distance $d_0$ in $V'$ from $\partial \D$ to $\partial B(0,\eta)$.
However, by Pre~Lemma~\ref{prelem:0}, $d_0$ is also equal to the distance in $U'$ from the core curve of $U$ to $\partial \D$, which leads to a contradiction. 
\end{proof}

\subsubsection{}

Given a hyperbolic open subset $U$ of $\C$, we will denote by $d_U(x,y)$ the hyperbolic metric on $U$ and by $B_U(z,r)$ the ball of center $z$ and radius $r$ for this metric.

Consider two virtual points $x$, $y$ and assume that all neighborhoods intersect. Let us prove that $x=y$. Let $(a,U,(\phi_n))$ be a representative of $x$ and $(b,V,(\psi_n))$ of $y$. We have to prove that there is a neighborhood of $a$ eventually contained in the domain of $\psi_n^{-1}\circ\phi_n$ and such that the latter converges on the neighborhood to an injective map sending $a$ to $b$.

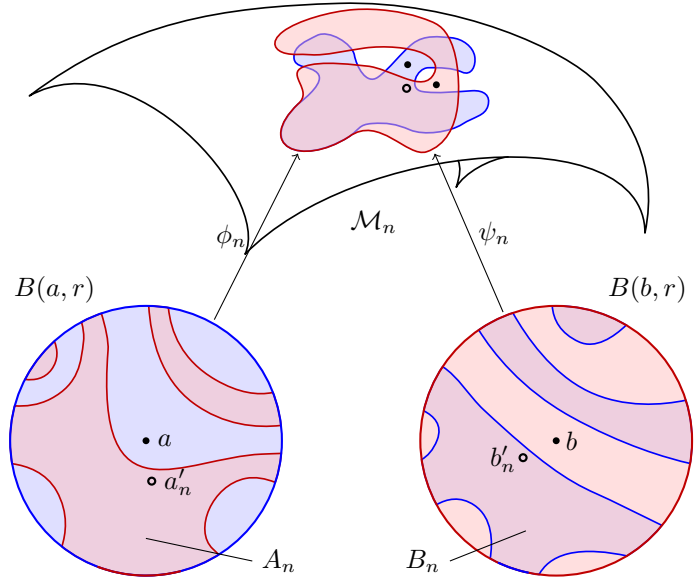
\begin{figure}
\begin{tikzpicture}
\node at (0,0) {\includegraphics[width=9cm]{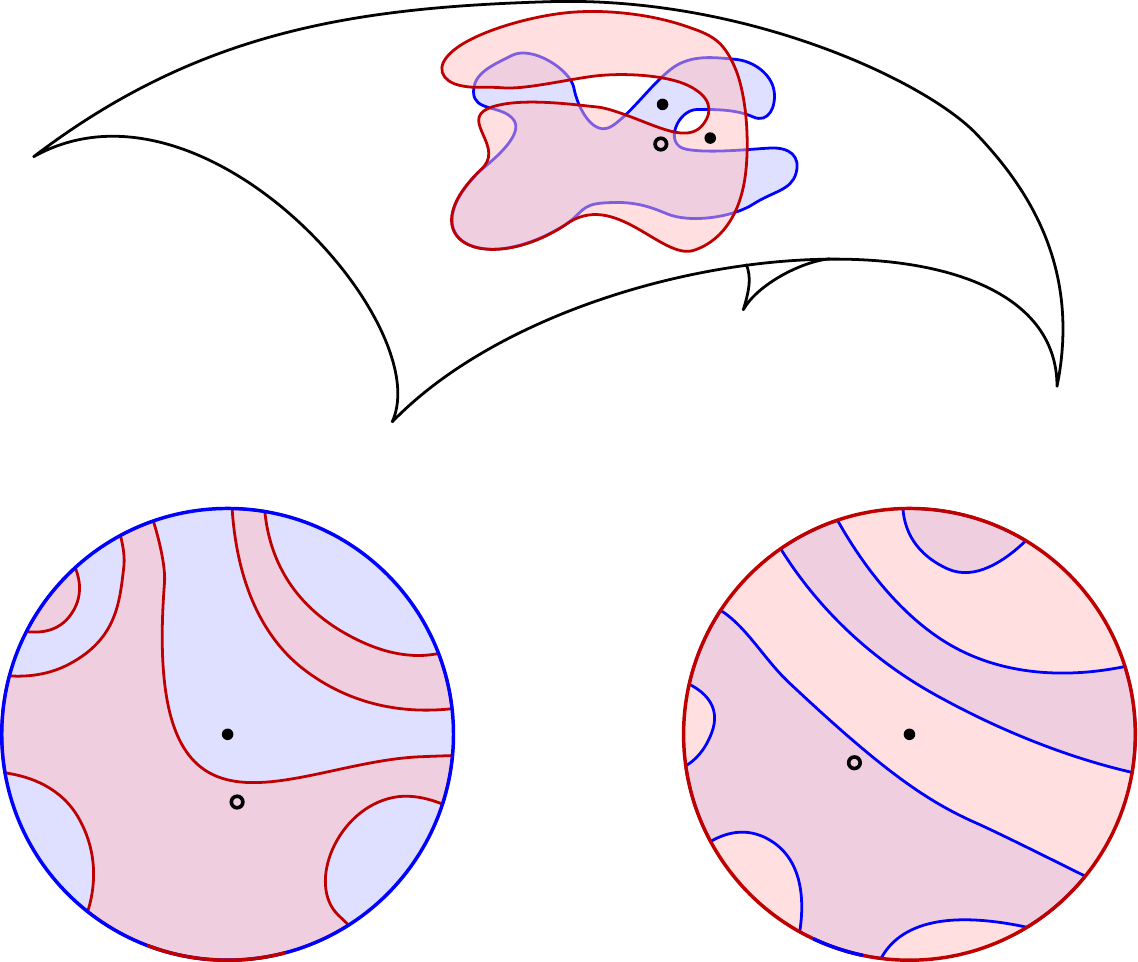}};
\node at (0.3,0.9) {$\cal M_n$};
\node at (-2.45,-2) {$a$};
\node at (-2.27,-2.55) {$a'_n$};
\node at (2.9,-2) {$b$};
\node at (2,-2.25) {$b'_n$};
\node at (3.9,0) {$B(b,r)$};
\node at (-3.9,0) {$B(a,r)$};
\draw[->] (-1.8,-0.4) -- node[left] {$\phi_n$} (-0.7,1.8);
\draw[->] (2,-0.3) -- node[right] {$\psi_n$} (1.1,1.8);
\draw[-] (-1.3,-3.6) node[right]{$A_n$} -- (-2.7,-3.3);
\draw[-] (1.3,-3.6) node[left]{$B_n$} -- (2.3,-3.2);
\end{tikzpicture}
\caption{Sketch of some objects in the proof of separation of $\frak M$ for Riemann surfaces. We want to prove in particular that $a$ or $b$ belongs to $A_n$ or $B_n$ and is well within. In this picture, neither does.}
\label{fig:sep}
\end{figure}

By hypothesis, for all $\epsilon>0$, there are points $a'\in B(a,\epsilon)\cap U_\infty$ and $b' \in B(b,\epsilon)\cap V_\infty$, such that $(a',U,(\phi_n))\sim (b',V,(\psi_n))$, i.e.\ the maps $\psi_n^{-1}\circ\phi_n$ converge in a neighborhood of $a'$ to an injective holomorphic map that sends $a'$ to $b'$.
\begin{proposition}\label{prop:3}
It is enough to prove that there is a neighborhood of $a$ eventually contained in the domain of $\psi_n^{-1}\circ\phi_n$. Similarly, it is enough to prove that there is a neighborhood of $b$ eventually contained in the domain of $\phi_n^{-1}\circ\psi_n$.
\end{proposition}
\begin{proof}
The second claim follows from the first by symmetry, using the fact that if $\phi_n^{-1}\circ\psi_n$ converges near $b$ to an injective map sending $b$ to $a$, recalling that those maps are holomorphic, then $\psi_n^{-1}\circ\phi_n$ converges near $a$ to an injective map sending $a$ to $b$.

Under the hypothesis of the first claim, consider a disk $B(a,\epsilon)$ contained in the neighborhood. Let $a'$ be as above. For $n$ big enough, the restriction of $\phi_n^{-1}\circ\psi_n$ is defined on $B(a,\epsilon)$ an univalent.
By properties of univalent maps, convergence near $a'$ implies convergence on all $B(a,\epsilon)$ to a univalent map. Moreover the limit maps $a'$ to $b'$. Now taking other pairs $a'$, $b'$ that are arbitrarily close to $a$ and $b$, we deduce by continuity that the limit sends $a$ to $b$.
\end{proof}

\begin{proposition}\label{prop:4}
It is enough to prove that there are neighborhoods $V$ of $a$ and $W$ of $b$ such that for all $n$ big enough, either $V$ is contained in the domain of $\psi_n^{-1}\circ\phi_n$, or $W$ is contained in the domain of $\phi_n^{-1}\circ\psi_n$.
\end{proposition}
\begin{proof} 
If one of the inclusion holds only for finitely many $n$, then we are done by the previous proposition. Otherwise, 
for those $n$ for which the second inclusion holds, the previous proposition implies that $\phi_n^{-1}\circ\psi_n$ converges on a neighborhood $W'\subset W$ of $b$, to an injective map sending $b$ to $a$. 
In particular, the range of $\phi_n^{-1}\circ\psi_n$ eventually contains a neighborhood of $a$. But this range is also equal to the domain of $\psi_n^{-1}\circ\phi_n$.
Hence the hypothesis of the previous proposition holds for \emph{all} $n$.
\end{proof}

The rest of the section aims at proving that the conditions of Proposition~\ref{prop:4} are satisfied.

\subsubsection{} Consider $r>0$ so that the closure of the disk $B(a,r)$ is contained in all $U_n$ with $n$ big enough, and similarly with $B(b,r)$ and $V_n$.
For such $n$, let $W_n=\phi_n(B(a,r))\cap\psi_n(B(b,r)) \subset \cal M_n$ and let $U'_n=\phi_n^{-1}(W_n)\subset B(a,r)$ and $V'_n=\psi_n^{-1}(W_n)\subset B(b,r)$. 
For $n$ big enough, let $A_n$ be the component of $U'_n$ that contains $a'$ and $B_n$ be the component of $V'_n$ that contains $b'$. A priori, these sets depend on $a'$ and $b'$.
Let $a'_n=a'$ or any sequence tending to $a'$ and $b'_n$ be the image of $a'_n$ by the change of variable $\psi_n^{-1}\circ\phi_n$. Then for $n$ big enough, $a'_n\in A_n$ and $b'_n\in B_n$.
The change of variable is a conformal bijection between $A_n$ and $B_n$ and extends continuously to their boundaries (recall that $B(a,r)$ and $B(b,r)$ are compactly contained in the domains of $\phi_n$ and $\psi_n$). It maps points on the boundary of $A_n$ that lie inside $B(a,r)$ to the boundary of $B(b,r)$ and similarly points on $B(b,r)\cap\partial B_n$ to $\partial B(a,r)$. See Figure~\ref{fig:sep}.

If $A_n$ is not simply connected for some $n$ then there is a component of $\partial A_n$ that lies inside $B(a,r)$ and it is mapped bijectively to $\partial B(b,r)$, hence all other points in $\partial A_n$ lie on $\partial B(a,r)$ and thus in fact $A_n$ is an annulus and one of its boundary component is $\partial B(a,r)$, $B_n$ is also an annulus with one boundary component being $\partial B(b,r)$. In this particular case, $\cal M_n$ is thus isomorphic to the Riemann sphere and $U'_n$ and $V'_n$ have only one component. See Figure~\ref{fig:sep-2}.

\subsubsection{} Let us first deal with the other case, when $A_n$ is simply connected. Let $r_0$ be given by Lemma~\ref{lem:1}. Choose any $r'_0<r_0$ Let $d'_0=d_\D(0,r'_0)$.
Recall that $a'$ and $b'$ are chosen after $\epsilon>0$, and belong to $B(a,\epsilon)$ and $B(b,\epsilon)$ respectively. Recall that $r>0$ is the radius of balls centered on $a$ and $b$ that are compactly contained in respectively $U_\infty$ and $V_\infty$.
Let us first fix $\epsilon=r'_0 \times r$ and then choose $a'$ and $b'$. Recall that $a_n\to a'$ and $b_n\tend b'$.
Because of the choices above, the $B(a,r)$-hyperbolic distance between $a$ and $a'$ is $<d_0$ hence the $B(a,r)$-hyperbolic ball of center $a'$ and radius $d_0'$ contains $a$. 
Similarly, the $B(b,r)$-hyperbolic ball of center $b'$ and radius $d_0'$ contains $b$. 

We claim that for all $n$ big enough, either $A_n$ contains $B_{B(a,r)}(a',d'_0)$ or $B_n$ contains $B_{B(b,r)}(b',d'_0)$. Otherwise 
by Lemma~\ref{lem:1} the harmonic measure within $A_n$ as seen from $a'_n$ of the part of the boundary of $A_n$ that lies inside $B(a,r)$ would be $>1/2$. And a similar thing could be said for $B_n$ and $b'_n$.
However, the change of variable being an isomorphism between $A_n$ and $B_n$ mapping $a'_n$ to $b'_n$, it preserves harmonic measure. We thus would get two disjoint sets with harmonic measure $>1/2$, which would be a contradiction.

Now recall that $B_{B(a,r)}(a', d'_0)$ contains $a$ and $B_{B(b,r)}(b', d'_0)$ contains $b$. Hence the hypotheses of Proposition~\ref{prop:4} are satisfied.

\begin{figure}
\begin{tikzpicture}
\node at (0,0) {\includegraphics[width=7cm]{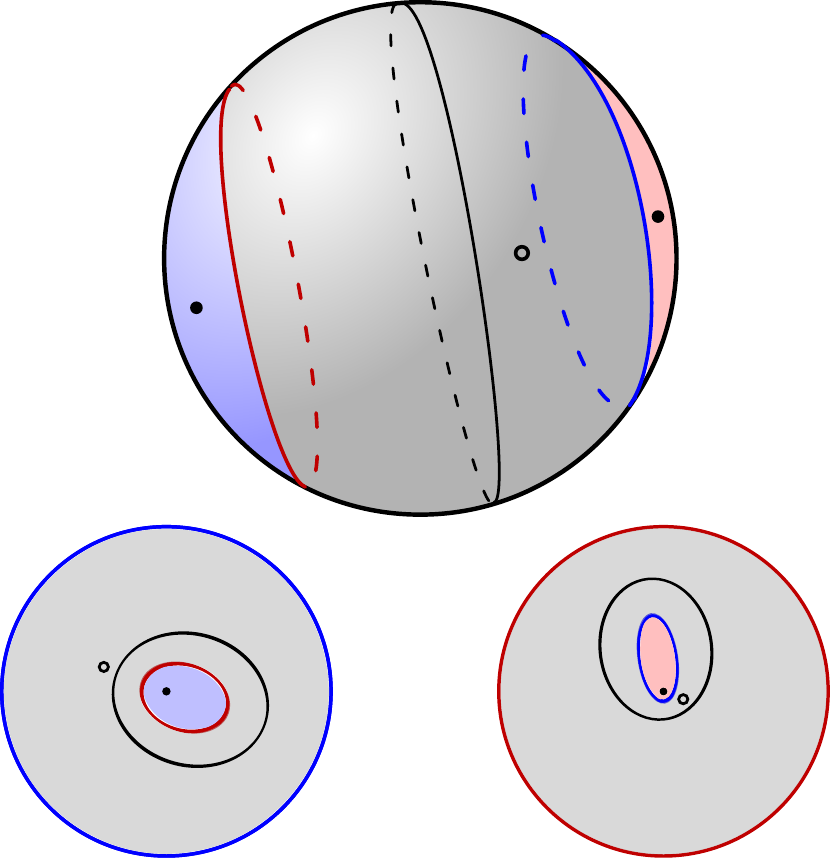}};
\node at (-2.4,2.9) {$\cal M_n$};
\node at (-1.9,-2.25) {$a$};
\node at (2.05,-1.95) {$b$};
\node at (-2.9,-2) {$a'_n$};
\node at (2.55,-2.45) {$b'_n$};
\node at (3,-0.6) {$B(b,r)$};
\node at (-3,-0.6) {$B(a,r)$};
\draw[->] (-1.4,-1) -- node[below right] {$\phi_n$} (-1.1,-0.5);
\draw[->] (1.4,-1) -- node[below left] {$\psi_n$} (1.1,-0.5);
\draw[-] (-0.9,-3.3) node[right]{$A_n$} -- (-1.8,-3.1);
\draw[-] (0.9,-3.3) node[left]{$B_n$} -- (1.9,-3.1);
\end{tikzpicture}
\caption{An annular case. Then $\cal M_n\simeq \wh\C$.}
\label{fig:sep-2}
\end{figure}

\subsubsection{} Let us now deal with the case when $A_n$ is not simply connected.
For convenience, let us denote $C_n=B(a,r)\setminus A_n$, and $D_n=B(b,r)\setminus B_n$.
Recall that $A_n$ and $B_n$ are annuli and that there is an isomorphism from $A_n$ to $B_n$ exchanging the outer and inner boundary curves, and sending $a'_n$ to $b'_n$.
The isomorphism must also map the core geodesic of $A_n$ to that of $B_n$.
Thus if $a'_n$ is exterior to this geodesic, then $b'_n$ is interior, and conversely.
Since the sequences $a_n$ and $b_n$ tend to limits $a'\neq a$ and $b'\neq b$, we deduce from Lemma~\ref{lem:3} that there is some $\epsilon>0$ such that for $n$ big enough, either $C_n \not\subset B(a,\epsilon r)$ or $D_n\not\subset B(b,\epsilon r)$. 

Now, similarly to the previous case, let us assume by contradiction that for infinitely many $n$, neither $A_n$ nor $B_n$ contains a fixed neighborhood of $a$ resp.\ $b$. So there are points in $C_n$ and $D_n$ tending to $a$ and $b$.
Since one of them also has point not in $B(a,\epsilon r)$ or $B(b,\epsilon r)$, the modulus of $A_n$ or $B_n$ cannot get too big. Since they are isomorphic, they have the same modulus. Therefore, there is in fact some $\epsilon'>0$ such that $C_n \not\subset B(a,\epsilon' r)$ \emph{and} $D_n\not\subset B(b,\epsilon' r)$.
By Corollary~\ref{cor}, for $n$ big enough the core curve of both annuli is eventually disjoint from $B(a \text{ or } b,\epsilon'')$ for some $\epsilon''>0$.
But, recall that there are other pairs $(a',b')$ such that $(a',U,\phi_n)\sim (b',V,\psi_n)$, and such that $a'$ and $b'$ are arbitrarily close to respectively $a$ and $b$. The corresponding sequences $a'_n$ and $b'_n$ would be both interior to the core curves, leading to a contradiction.

So $A_n$ eventually contains $B(a,\rho)$ for some $\rho>0$ or an analogous statement holds for $B_n$. We can thus conclude using Proposition~\ref{prop:4}.

\bigskip

This ends the proof of separation of $\frak M$ in the case of Riemann surfaces.

\bibliographystyle{alpha}

\end{document}